%% file: sidedcombine-main2.tex
\documentclass[option]{amsart}
\usepackage{amssymb, latexsym, mathtools, amsthm}
\begingroup
	\makeatletter
        \@for\theoremstyle:=definition,remark,plain\do{%
        \expandafter\g@addto@macro\csname th@\theoremstyle\endcsname{%
        \addtolength\thm@preskip\parskip}}
\endgroup
\renewenvironment{abstract}
 { \normalsize
  \list{}{\setlength{\leftmargin}{.0cm}%
    \setlength{\rightmargin}{\leftmargin}}
  \item {\bf \abstractname.} \relax}
{\endlist}

\usepackage{newtxmath}

\usepackage{natbib}

\newcommand{\lce}{left-c.e.\ }

\newcommand{\ce}{c.e.\ }
\newcommand{\lcen}{left-c.e.}

\newcommand{\pf}{prefix-free }

\newcommand{\ml}{Martin-L\"{o}f }

\newcommand{\twomel}{2^{<\omega}}

\newcommand{\mm}{\mathbf{m}}

\newcommand{\MM}{\mathcal{M}}

\newcommand{\Mb}{\mathbf{M}}

\usepackage{enumerate}
\usepackage{graphicx}
\usepackage{subfigure}
\usepackage{rotating}
\usepackage{epstopdf}
\usepackage{tikz-cd}
\usepackage{caption}
\usepackage{multirow}
\usepackage{thm-restate}
\usepackage{bm}
\definecolor{dnrbl}{rgb}{0,0,0.3}
\definecolor{dnrgr}{rgb}{0,0.3,0}
\definecolor{dnrre}{rgb}{0.5,0,0}
\usepackage[colorlinks=true, citecolor=dnrgr, linkcolor=dnrre, urlcolor=dnrre] {hyperref}
\usepackage{cleveref}

\newtheorem{theorem}{Theorem}[section]

\newtheorem{lemma}[theorem]{Lemma}
\newtheorem{fact}[theorem]{Fact}
\newtheorem{claim}[theorem]{Claim}

\theoremstyle{definition}
\newtheorem*{strategyn*}{Strategy}
\newtheorem{definition}[theorem]{Definition}

\newtheorem{example}[theorem]{Example}

\newtheorem{strategy}[theorem]{Strategy}
\theoremstyle{remark}
\newtheorem{remark}[theorem]{Remark}

\newtheorem*{question*}{Question}
\newtheorem*{theorem*}{Theorem}

\newcommand{\uhr}{{\upharpoonright}}

\def\t{\tilde}
\def\m{m}
\def\var{\varepsilon}
\def\mcal{\mathcal}
\def\mm{A}
\def\linee{line}
 \def\betting{betting}
 \def\symm{Sym}
\def\avg{avg}
\def\upright{up-right}
\def\invoking{invoking}
\def\sideindicator{side-indicator}
\def\validtuple{valid tuple}
\def\mbE{\mathbb{E}\hspace{0.04cm}}
\def\variance{Var\hspace{0.04cm}}
\def\catching{catching}
\def\semivalid{semi-valid}
\def\variancee{variance}
\def\constant{constant}
\def\potentialwin{potentially-winning}
\def\winningstring{non-catchable-string}
\def\supermartingaleapp{supermartingale-approximation}
\def\homogeneous{homogeneous}
\def\uhr{\upharpoonright}

\title{Irreducibility of enumerable betting-strategies}
\author{George Barmpalias
 \and Lu Liu}
\thanks{Supported by NSFC grant 11971501.}


\begin{document}
\maketitle

\input{marcro.tex}

\begin{abstract}
We study the problem of whether a betting strategy
can be reduced to a  set of restricted  betting strategies, such as betting on a restricted
set of stages or bet on a restricted of favorable outcomes.
We show that the class of effectively enumerable betting strategies
has irreducible members (which cannot be reduced to a set of restricted betting strategies).
We answer   questions of Kastermans and Hitchcock by constructing a real on which no kastergale
(\lce supermartingale with effectively determined favorable outcomes) succeeds, but some general
\lce supermartingale succeeds on it.
We also generalize Muchnik's paradox  by showing that
there is a non-\ml-random real such that no muchgale
(\lce supermartingale betting on restricted stages) succeeds on it.
Our methodology is general enough to strongly support the metaconjecture that
if a natural class of \lce supermartingales defines \ml randomness,
then a single member of that class can do so;
thus,
 the class of \lce supermartingales cannot be reduced to a simpler, natural subclass of it.
 For example, we show that there is a non-\ml-random real such that no kastergale or muchgale succeeds on it.
 \end{abstract}
\vspace*{\fill}
\noindent{\bf George Barmpalias}\\[0.2em]
\noindent State Key Lab of Computer Science,
Institute of Software, Chinese Academy of Sciences, Beijing, China.
\textit{E-mail:} \texttt{\textcolor{dnrgr}{barmpalias@gmail.com}}.\par
\addvspace{\medskipamount}\medskip\medskip
\noindent{\bf Lu Liu}\\[0.2em]
\noindent School of Mathematics and Statistics, HNP-LAMA, Central South University,
Changsha, China.
\textit{E-mail:} \texttt{\textcolor{dnrgr}{g.jiayi.liu@gmail.com.}}
\vfill \thispagestyle{empty}
\clearpage
\setcounter{tocdepth}{2}
\tableofcontents

\section{Introduction}
In  betting games, a player is repeatedly allocates capital on an array of options,
at the risk of loss or the possibility of gain, depending on the outcomes of an unpredictable process.
Such games have proved to be very useful models of a variety of problems, and an indispensable tool for understanding
the nature of information that has its roots in
 frequentist  foundation for probability by
 \citet{MR21400mises,misesbook}.
A betting strategy for the player has two components: (a) the favored outcome at each bet; and (b) the amount of capital,
the {\em wager} that is placed on the preferred outcome, which reflects the odds of the bet with respect to the information
that is available to the player. For simplicity we restrict our attention to the generic case of binary outcomes, where the wager
is doubled or lost depending whether the actual outcome was favored or not.
Typical restricted betting strategies include the self-explanatory {\em fixed-wager}, the {\em martingale}
where each wager is twice the previous one, and its variations,
and {\em proportional betting} in which wagers are set as a fixed proportion of the existing capital.
A prominent example of the latter is the ubiquitous investment strategy known as the {\em Kelly criterion},
which maximizes the expectation of the  logarithm of the capital, and was
discovered by
\cite{kellyorig}
as an interpretation of
Shannon's concept of information rate.

The success of a betting strategy may lie on its capital allocation, its dynamic choice of
favorable outcomes, a strategic choice of the stages where a bet is placed, or  a combination of the above.
In order to examine the dependence of a betting strategy  on the above components, we can ask if it can be \emph{reduced}
\footnote{Here
``reduce" means that it succeeds on a real $x\in 2^\omega$ iff at least one of the restricted betting strategies succeeds on $x$.} to
a set of
restricted betting strategies,
which do not alternate favorable outcomes, or restrict their wagers on a specific subset of the stages:
\begin{enumerate}[\hspace{0.5cm}(a)]
\item can it be reduced  to {\em fixed-outcome} betting strategies ?
\item can it be reduced to betting strategies that bet on alternate stages?
\end{enumerate}
Property (a)
indicates that the success  is not based on strategic alternations of favorable outcomes, but  on
its dynamic capital allocation.
The possibility of such betting strategy reduction  is the topic of the present work;  the answer
depends on the information that is available in the game. We show that
\begin{equation*}
\parbox{10cm}{there are enumerable betting strategies
that are not reducible to \emph{restricted} enumerable betting strategies}
\end{equation*}
where restricted will be made precise  and includes the above examples of fixed-outcome and stage-restricted strategies.
In this way, we answer a question of Kastermans and a question of Hitchcock, as reported by
\citet[\S 7.9]{rodenisbook} and \citet{rodkasternewt},
as well as extending  work of \citet{muchnik2009algorithmic}.

\ \\

\noindent\textbf{Kasterman's question}.
Following \citet{Ville1939}
 we formalize  betting strategies in terms of {\em martingales}:
non-negative  functions $M$ on $2^{<\omega}$  such that
\begin{equation}\label{vvqVoKIn9m}
2\cdot M(\sigma)=M(\sigma\ast 0)+M(\sigma\ast 1)
\hspace{0.5cm}\textrm{for all $\sigma\in 2^{<\omega}$}
\end{equation}
where $\omega$ is the set of natural numbers,
$2^{<\omega}$ is the set of binary strings
and $\ast$ denotes concatenation. Martingales model the capital of a player who plays against a casino, by
placing wagers on binary outcomes. If the equality in \eqref{vvqVoKIn9m} is replaced with $\geq$
we get the notion of {\em supermartingale}, which often model inflation environments
as well as savings. In terms of feasibility, we also need to qualify the betting strategies
from a constructive point of view: a standard choice advocated by \citet{Schnorr:71} is to consider computable betting strategies, namely
martingales that can be fully defined by a Turing machine, and
their effective countable mixtures. The latter can be
equivalently defined as the
{\em \lce (left-computably-enumerable) martingales, supermartingales}, which means that there exist
computable array $(M[t]:t\in\omega)$ of
(super)martingales such that $M[t]$ dominates $M[t-1]$ and their point-wise limit is $M$.
Left-c.e.\ supermartingales represent games where betting decisions  are not fully transparent to
a constructive observer,
being infinitary and revealed only gradually, through monotone approximations that converge to the limit wagers.
We say that a (super)martingale $M$ {\em succeeds} on an\
infinite binary sequence $x$ (also called a {\em real}) if it is unbounded on the initial segments of $x$.
In this case we may say that $x$ is {\em predictable with respect to $M$}.
Left-\ce martingales and supermartingales are indispensable tools in the study of algorithmic information:
$x$  is algorithmically random in the sense of \citet{MR0223179}, also known as {\em \ml-random } or {\em 1-random},
iff no \lce (super)martingale succeeds on it.

We formally define fixed-outcome betting strategies and their generalizations.
\begin{definition}[Sided betting strategies]\label{VPTxEMxdFT}
Let $i<2$,  $M:\twomel\to\mathbb{R} $ and let $p$ be a partial boolean function
on binary strings. We say that $M$ is {\em $i$-sided at $\sigma$} if
$M(\sigma\ast i)\geq M(\sigma\ast (1-i))$.
We say that {\em $M$ is $p$-sided} if it
is $p(\sigma)$-sided on each $\sigma$ in the domain of $p$,
and $M(\sigma\ast i)= M(\sigma\ast (1-i))$ otherwise.
For a $p$-sided function $M$, we say $M$ is
\begin{itemize}
\item {\em single-sided} if $p$ is a boolean constant;
\item {\em partially-computably-sided} if   $p$ is partial
computable
\footnote{Clearly, for every left-c.e. partially-computably-sided supermartingale
$M$, there is a computable array of supermartingales $(M[t]:t\in\omega)$
and a partial computable function $p:\subseteq 2^{<\omega}\rightarrow 2$
such that $\lim\nolimits_{t\rightarrow\infty}M[t](\sigma) = M(\sigma)$
for all $\sigma\in 2^{<\omega}$ and $M[t]$ is $p[t]$-sided
where $p[t]$ is the approximation of $p$ at time $t$.
 i.e., For each $\sigma\in 2^{<\omega}$,
 $M$ has only one chance to decide its \sided ness at $\sigma$
 and before it makes that decision, it has to be both $0,1$-\sided\ at
 $\sigma$.}.
\end{itemize}
\end{definition}

The interplay between the flow of information in the game and the strength of a betting strategy has been the core of several problems in
algorithmic information, most notably the question of whether computable {\em non-monotonic betting} can simulate \lce
supermartingales, which has its roots in the study of
non-monotonic selection rules by  \citet{Kolmogorov98} and  \citet{Loveland66} and
remains open to this day despite numerous attempts by \citet{Merkleea06}, \citet{Bienvenu.Hoelzl.ea:09}
among others.
In this spirit, Kastermans asked the following
generalized version of (a) for \lce supermartingales:
\begin{equation}\label{eYV1hawDh}
\parbox{10cm}{{\em Kastermans' question:} given a \lce supermartingale $M$ which succeeds on $x$, does there exist
a  partially-computably-sided $M'$ which succeeds on $x$ ?}
\end{equation}
Hitchcock asked if the above reduction holds in the special case where the wagers the ratios $M(\sigma\ast i)/M(\sigma)$
are uniformly \lcen, which means that a commitment of a betting strategy to bet a certain proportion of the capital on some outcome
at some stage of the approximation is irreversible, and unaffected by any potential future increase of the existing capital.
This subclass of the \lce supermartingales is known as {\em hitchgales}, while the
partially-computably-sided \lce supermartingales are known as {\em kastergales}. These problems were
reported by  \citet[\S 7.9]{rodenisbook} and \citet{rodkasternewt}.

\subsection{Our contribution}
Our main technical result is a negative answer to the questions of Kastermans and Hitchcock, along with several extensions
which give negative answers to questions like (a), (b) about reducibility into restricted betting strategies, for the case of \lce supermartingales.
We reach the solution to \eqref{eYV1hawDh} gradually through a systematic
analysis of the expectation and variance of supermartingales in a game-theoretic framework, which differentiates our approach
to previous attempts that are detailed below.
Our main innovation is developed in \S\ref{singsec4} where
we prove the irreducibility to the fixed-outcome betting strategies:
\begin{theorem}[Fixed-outcome]\label{singth3}
There exists a real $x$ such that
no single-sided \lce supermartingale succeeds
on $x$, but $x$ is not 1-random,
so some \lce supermartingale succeeds on $x$.
\end{theorem}
In \S\ref{singsecpartial} we obtain a negative answer to the questions of
Kastermans and Hitchcock.
\begin{theorem}\label{singth1}
There exists a real $x$ such that
no kastergale succeeds
on $x$, but $x$ is not 1-random, so some \lce supermartingale succeeds on it.
\end{theorem}
We demonstrate the generality of our methodology in \S\ref{singsecgen}, by generalizing
a result of \citet{muchnik2009algorithmic} that  is known as {\em Muchnik's paradox}
and has been discussed by \citet{ChernovSVV08}  and   \citet{stacsBauwens14}
in the context of {\em online complexity} and prediction.
Consider the following restriction of \lce supermartingales,
where they are required to restrain their betting to certain partitions of the stages.
\begin{definition}[muchgale]
\label{singdefmuchgale}
Given $i<l$, we say that $M$ is \emph{$(l,i)$-\betting}  if
\[
|\sigma|\equiv i\mod l
\hspace{0.3cm}\Rightarrow\hspace{0.3cm}
M(\sigma ) \geq \max_{j<2} M(\sigma\ast j).
\]
We refer to $(l,i)$-betting \lce supermartingales as \emph{muchgales}.
\end{definition}
\citet{muchnik2009algorithmic} considered the
$(2,i)$-betting \lce supermartingales and showed that there exists a real $x$ such that
no $(2,i)$-betting \lce supermartingale succeeds on it,
but $x$ is not 1-random.
In \S\ref{singsecgen} we generalize this result and show that
some reals are unpredictable with respect to both {\em all muchgales} and all kastergales,  yet they are not 1-random:
\begin{restatable}{theorem}{singthunion}\label{singth5}
There exists a real $x$ such that
no kastergale or muchgale succeeds
on $x$, but $x$ is not 1-random, so some \lce supermartingale succeeds on it.
\end{restatable}
Finally we put forth and give strong evidence toward the following
\footnote{A class of left-c.e. supermartingales defines
1-randomness iff
for every non-1-random real $x$, there is some member in the class
succeeding on $x$.}
\begin{equation}\label{singconj0}
\parbox{10.5cm}{{\bf Conjecture:} if a
natural subclass of left-c.e. supermartingales  defines
 1-randomness, then   a single member
of that class can do so.}
\end{equation}
The definition of  ``natural'' will be given
in section \ref{singsecgen}.
The kastergale, hitchgales as well as muchgales are examples of natural subclasses
of \lce\ supermartingales.
A consequence of (\ref{singconj0})
is that \lce supermartingales cannot be reduced to
restricted  supermartingales; that is:  no simpler
\footnote{Here ``simpler" means no single member of the class succeeds on all non-1-random reals.} natural subclass of \lce supermartingales defines 1-randomness.

\subsection{Previous work}
Our irreducibility results are interesting because they only occur in the context of enumerable strategies. Indeed,
every computable martingale is the product of:
\begin{enumerate}[\hspace{0.3cm}$\bullet$]
\item  two computable fixed-outcome martingales
\citep[\S 3]{barmpalias2020monotonous}
\item two computable martingales that bet on alternate stages
\citep[\S 2]{muchnik2009algorithmic}
\end{enumerate}
so in a fully constructive environment, our initial questions (a), (b) have positive answers.
\citet{muchnik2009algorithmic} was the first to demonstrate a negative answer, by
showing that there exists a real $x$ which is unpredictable with respect to \lce supermartingales
that only bet on even stages and those which only bet on odd stages, but is not 1-random, so some
\lce supermartingale succeeds on it. Hence such restricted betting strategies do not
characterize 1-randomness. \citet{stacsBauwens14}  provides expressions of Muchnik's paradox in terms of online complexity.

\citet{barmpalias2020monotonous} examined the case of single-sided strategies, which turned out to be more powerful than
muchgales.  Using different methods, they obtained a weak analogue of Theorem \ref{singth3},
which is restricted to effective countable mixtures of computable martingales; this is  a subclass of \lce
martingales that corresponds to the condition that the wagers are  \lce in a uniform way.
The dependence of their method to the martingale condition and the enumerability of the wagers
is discussed by \citet[\S 5]{barmpalias2020monotonous} where the need for a new method
toward an answer to Kasterman's question is anticipated.
\citet{muchpar22} use different methods to deal with similar questions in terms of fractal dimension.

The problem of restricted betting has also been studied by \citet{schweinsberg_2005}
in a purely probabilistic fashion, in the context of {\em red and black} games of \citet{howgamble}.
A distinguished characteristic of our line of work with the probabilistic tradition is the enumerability of the strategies, which
corresponds to imperfect information during the game. While variance analysis is a standard tool in probability, our
methodology is novel in the study of effective strategies as all of the previous works relied on recursion-theoretic methods.

\subsection{Organization}
The rest of this paper is devoted to proving Theorems \ref{singth3} - \ref{singth5}.
Theorems \ref{singth3},   \ref{singth1}, \ref{singth5} are proved in
\S \ref{singsec4},  \S \ref{singsecpartial},
\S \ref{singsecgen} respectively.
In \S  \ref{singsecgen}, we demonstrate
strong evidence toward  \eqref{singconj0}.
The proof of Theorem \ref{singth3} in \S\ref{singsec4} already
includes the bulk of our method, which is then used for
Theorems \ref{singth3} - \ref{singth5}.
Many lemmas, claims and definitions in \S   \ref{singsecpartial},
are analogues of those in \S  \ref{singsec4}. We repeat them in order to make
it convenient for readers to check the proof.

\subsection{Notation}\label{jKb66xtKze}
We maintain some consistency in our notation, reserving variables
$i, j, k, n, m, t,$ for non-negative integers, $\sigma,\tau,\rho$
for finite strings, $x,y,z$ for reals (infinitely long binary sequence), 
and $A,B$ for sets of reals or strings.

For   $\sigma,\rho\in 2^{<\omega}$:
\begin{itemize}
\item $\rho\preceq\sigma, \sigma\succeq\rho$ denote that $\rho$ is a prefix of $\sigma$;
\item let $[\sigma]^\preceq $ be the set of binary strings extending $\sigma$;  for a set $A\subseteq 2^{<\omega}$, let
$[A]^\preceq := \cup_{\sigma\in A}[\sigma]^\preceq$;
\item let    $[\sigma]:=\{x\in 2^\omega: \sigma\preceq x\}$
and   $[A] = \cup_{\sigma\in A}[\sigma]$;
\item   let $\sigma\uhr n$ denote the string $\sigma(0)\cdots\sigma(n-1)$;
let $\sigma\uhr_m^n$ denote the string $\sigma(m)\sigma(m+1)\cdots\sigma(n)$;
\item let $\emptyset$ also denote the empty string.
\end{itemize}
For a set $A\subseteq 2^{<\omega}$, the measure of $A$,
denoted as $\m(A)$, refers to
the Lebesgue measure of $[A]$;
for $\sigma\in 2^{<\omega}$, we write $\m(\sigma)$ for $\m(\{\sigma\})$;
we write $\m(A|B)$ for $\m(A\cap B)/\m(B)$.
For a vector $\mb{v}\in\mathbb{R}^k$, we use
$||\mb{v}||_p$ to denote the $L^p$-norm of  $\mb{v}$,
namely $\sqrt[p]{\sum_{j<k}|\mb{v}(j)|^p}$.

 For functions $M,\hat M$,
 \begin{itemize}
 \item  we say $\hat M$ \emph{dominates} $M$ on set $A$
 if $\hat M(\sigma)\geq M(\sigma)$ for all $\sigma\in A$;
\item we denote the average of $M$ on a   set $A$ of strings by:
\[
\int_{\sigma\in A}M(\sigma):=\int_A M := \sum_{\sigma\in A}M(\sigma)\cdot \mu(\sigma).
\]
\end{itemize}

\input{singlesidedcombine.tex}
\input{partialsided5.tex}

\input{generalizebetting3combine.tex}

\bibliographystyle{abbrvnat}
\bibliography{sided}

\end{document}

%% file: marcro.tex
\def\sided{sided}
\def\mcal{\mathcal}
\def\h{\hat}
\def\t{\tilde}
\def\tba{\textcolor{red}{TBA}}

\def\m{\mu}
\def\mb{\mathbf}
\def\var{\varepsilon}
\def\mm{A}
\def\linee{line}
\def\alice{Alice}
\def\baby{Baby}
 \def\betting{betting}
 \def\symm{Sym}
\def\avg{avg}
\def\upright{up-right}
\def\v{\vec}
\def\invoking{invoking}
\def\sideindicator{side-indicator}
\def\validtuple{valid tuple}
\def\mbE{\mathbb{E}}
\def\variance{Var}
\def\catching{catching}
\def\semivalid{semi-valid}
\def\typeaattention{winning attention}
\def\variancee{variance}
\def\constant{constant}
\def\potentialwin{potentially-winning}

\def\catchingpoint{catching-point}
\def\winningstring{non-catchable-string}

\def\supermartingaleapp{supermartingale-approximation}
\def\homogeneous{homogeneous}
\def\homogeneity{homogeneity}
\def\stronghomogeneous{homogeneous}
\def \stronghomogeneity{homogeneity}
\def\uhr{\upharpoonright}

\def\fin{Fin}
\def\treelike{tree-like}

%% file: singlesidedcombine.tex
\section{Defeat   single-\sided\ supermartingales}
\label{singsec4}

As many assertions in computability theory,
Theorem \ref{singth3} can be seen as a game between two players
(hence called \alice\ and \baby) where
\alice\ constrols a Martin-L\"{o}f test and
\baby\ controls infinitely many
$i$-\sided\ left-c.e. supermartingales.
It is usually very helpful to study the finite version
of such games.
In the game, \alice\ tries to provide
(enumerate) a $\sigma^*$
so that   $\sum_{i<2} M_i(\h\sigma) $ does not reach certain threshold
for all $\h\sigma\preceq\sigma^*$
(where $M_i$ is the left-c.e. $i$-\sided\ supermartingale
controled by \baby).
This gives rise to
the following finite game.
Let $0\leq c \leq d,n\in\omega$.

 \begin{definition}[$(c,d,n)$-\sided-game]
 \label{singdef2}

At each round $t\in\omega$, \alice\ firstly enumerates a $\sigma\in 2^n$
(that has not been enumerated before);
then \baby\ presents $i$-\sided\ supermartingale
$M_i[t]$ (for $i<2$) such that the following hold:
let $\Mb[t] := (M_0[t],M_1[t])$,
\begin{itemize}

\item $||\Mb[t](\h\sigma)||_1\geq d $ for some $\h\sigma\preceq\sigma$;
\item $M_i[t]$ dominates $ M_i[t-1]$ on $2^{\leq n}$
(we set $M_{i}[-1]\equiv 0$).

\end{itemize}
\alice\ wins the game if for some $t$, $||\Mb[t](\emptyset)||_1 \geq  c$.
\end{definition}

We call $\h\sigma$ a \emph{\catchingpoint} of $\sigma$.
Let $A$ be the set of $\sigma$ that \alice\ enumerated
during the game \ref{singdef2}.
We show that  given $c<d$, $\varepsilon>0$, \alice\ can win the $(c,d,n)$-\sided-game \ref{singdef2}
with a cost $\m(A)\leq \varepsilon$ (for some $n$).
By scaling,
\begin{align}\nonumber 
\bullet\ \  \text{$(c,d,n)$-\sided-game is equivalent to $(c\h c, d\h c,n)$-\sided-game
(for all $\h c>0$).}
\end{align}
When $d= 1$ the  $(c,d,n)$-\sided-game
is called  $(c,n)$-\sided-game.
We prove

\begin{lemma}\label{singlem41}
For every  $  c<1, \varepsilon>0$, there is an $n\in\omega$ such that
 \alice\ has a winning strategy  for the
$(c,n)$-\sided-game \ref{singdef2} such that $\m(A)\leq \varepsilon$.
\end{lemma}

We emphasize that there is no complexity notion involved in
game \ref{singdef2} or Lemma \ref{singlem41}.
The proof of Lemma \ref{singlem41} consists of several separate ideas listed below
(which will consist most of the ingredients needed for Theorem \ref{singth1}).
The first three ingredients reduce Lemma \ref{singlem41} to a weaker
assertion (these three ingredients
reduce the winning strategy  of  game \ref{singdef2} to a series
of other games with each being  easier for \alice\ to win than the previous one).

\begin{enumerate}[1.]
\item {\em Nesting argument.}
This is similar
as  how we apply Lemma \ref{singlem41} iteratively
to prove Theorem  \ref{singth3}.
This, roughly speaking,
allows \alice\ to enumerate much more $\sigma$
in trade of a reasonably stronger goal (i.e., $c$ is
increased reasonably while $\m (A)$ is allowed to be
close to $1$).
The nesting argument reduce
 Lemma \ref{singlem41} to Lemma \ref{singlem42}.
\item {\em Dynamic goal argument.} A dynamic goal version of game \ref{singdef2}
has the same rule as game \ref{singdef2} but a different winning criterion
(of \alice). The dynamic winning criterion allows
\alice\  to achieve a   dynamic goal
 depending on how much resource she cost, namely $\m(A)$
 (the more she cost the better goal she is supposed to achieve).
The argument shows that a winning strategy for the dynamic goal version
of game \ref{singdef2} (namely game \ref{singdef3}) gives rise to a winning strategy of game \ref{singdef2}.
 Thus
 Lemma \ref{singlem42} is subsequently
 reduced to Lemma \ref{singques0}.
\item {\em Imposing restriction on the \sided\ player.}
The dynamic goal
game (game \ref{singdef3}) will allow us to impose further restrictions on \baby's
action: $||\Mb[t](\sigma)||_1\geq 1$
(instead of $||\Mb[t](\h\sigma)||_1\geq 1$ for some $\h\sigma\preceq\sigma$) and
$||\Mb[t](\rho)||_1\leq 1+\delta$ for all $\rho\in 2^{\leq n}$
(where $\delta$ is a pre-chosen small parameter).
These restrictions force that the more capital \baby\
allocates, the more $M_i[t]$  is alike a martingale.
Thus we reduce   game \ref{singdef3}
to its restricted version (game \ref{singdef4}) and the winning strategy
of game \ref{singdef3},
namely Lemma \ref{singques0}, is reduced to
Lemma \ref{singlem43}.

\item A winning strategy for the restricted dynamic game
\ref{singdef4} (proof of Lemma \ref{singlem43}).
\end{enumerate}
The following diagram illustrates the framework of the proof
\footnote{Where  $\psi \overset{\varphi}{\hookrightarrow }
\phi$ means assertion $\varphi$ reduce
assertion $\psi$ to $\phi$. i.e., $\varphi\wedge \phi$ implies $\psi$.}.
\begin{align}\nonumber
 &\text{Lemma \ref{singlem41}}
 \overset{Claim \ref{singclaim1}}{\hookrightarrow }
 \text{Lemma \ref{singlem42}} (\text{more resource allowed,
 \S  \ref{singsec2}});\\ \nonumber
  &
  \text{Lemma \ref{singlem42}}\overset{Claim \ref{singclaim2}}
  {\hookrightarrow }\text{Lemma \ref{singques0}
  (dynamic goal, \S  \ref{singsec10})};
  \\ \nonumber
  &\text{Lemma \ref{singques0}}\overset{Claim \ref{singclaim4}}{\hookrightarrow }
  \text{Lemma \ref{singlem43}
  (imposing restrictions on \baby,
  \S  \ref{singsec1})
  };
  \\ \nonumber
  &\text{Proof of Lemma \ref{singlem43}
  (section \ref{singsec0})}.
\end{align}

Notably, other than the proof of Lemma \ref{singlem43}, none of these
arguments take advantage
of $i$-\sided ness but the mere fact that $M_i[t]$ is supermartingale
and nondecreasing with respect to $t$.
The rest of \S  \ref{singsec4} is devoted to prove Lemma \ref{singlem41}.
Before that, we prove Theorem \ref{singth3}.

\begin{proof}[Proof of Theorem \ref{singth3} using Lemma \ref{singlem41}]
Note that given  $i$-\sided\ left-c.e.    supermartingale $M_i$,
let $\Mb = (M_0,M_1)$,
if  $||\Mb(\emptyset)||_1<c$ and $c<d$, then for some $n\in\omega$,
we can (effectively) play the winning strategy for the $(c,d, n)$-\sided-game \ref{singdef2} on $2^n$
(against $M_0,M_1$)
to computably enumerate a sequence
  $V=(\sigma[0],\cdots,\sigma[s-1])$ of strings
  (with $\m(V)\leq 1/2$) such that let $\sigma^*=\sigma[s-1]$,
we have $||\Mb(\h\sigma)||_1\leq d$ for all $\h\sigma\preceq\sigma^*$.
This is clearly done by taking $M_i[t]$ as the $i$-\sided\ supermartingale presented
by \baby. We refer to this process as  \emph{controlling
 $M_0, M_1$ on $2^{[0,n]}$} and $\sigma^*$ as a \emph{\winningstring}.
Roughly speaking,  the non-1-random real $x$ is produced as follows.

\begin{enumerate}[1.]
\item On $2^{n_0}$, play the winning strategy for  \sided-game
\ref{singdef2} to control $M_0,M_1$ on $2^{[0,n_0]}$.

\item Upon obtaining $\sigma^*_0$, play the winning
strategy for the \sided-game \ref{singdef2} (on $[\sigma^*_0]^\preceq\cap 2^{n_0+n_1}$) to
control $M_0+\h M_0, M_1+\h M_1$ on $2^{[n_0,n_1]}$ and so forth.
Where $\h M_i$ is (some  scaling) of another $i$-\sided\ left-c.e. supermartingale.
\end{enumerate}
In the end, the non-1-random $x$ will be $\cup_k\sigma_k^*$.
Of course, when the winning strategy tells us to enumerate $\sigma\in 2^{n_0}$,
we have no idea whether $\sigma$ is a \winningstring.
We will simply take $\sigma$ as if it is (unless found otherwise later)
and keep doing the above process (item (2)) over $\sigma$. If at some point it is found that
$\sigma$ is not (a \winningstring), then go back to the \sided-game \ref{singdef2} on $2^{n_0}$
(give up whatever is done over $\sigma$),
 obtain (invoking the winning strategy) the next $\t \sigma\in 2^{n_0}$ and proceed with
$\t\sigma$.
The more concrete argument goes as follows.

Let  $n_0,n_1,\dots$,
$0<c_0< d_0<c_1<\cdots <2 $ be
computable sequences such that:
\begin{align}\nonumber
\bullet\ \ &\text{there is a winning strategy of \alice\ for the
$(c_k, d_k, n_k)$-\sided-game }\\ \nonumber
&\text{such that $\m(A)\leq 1/2$ for all $k$.}
\end{align}
Let $M_{i,0},M_{i,1},\cdots$ be an effective list of all  $i$-\sided\ left-c.e.   supermartingale;
let $\Mb_k = (M_{0,k},M_{1,k})$.
For convenience, suppose
$$||\Mb_{k}(\emptyset)||_1\leq \delta_k
\text{ for all $k\in\omega,i\in2$}.
$$
where $(\delta_k\in \mathbb{Q}:k\in\omega)$ is a computable sequence so that each is sufficiently small
(to  be specified).

We will define a sequence $(V_k\subseteq 2^{n_0+\cdots+ n_k}: k\in\omega)$ of uniformly
c.e. sets with $\m(V_k)\leq 2^{-k}$;
 and the non-1-random real $x$ will be (the unique) element of $\cap_k [V_k]$.
To make sure that no $i$-\sided\ supermartingale succeeds on $x$, we satisfy the requirement:
\begin{align}\label{singeqproofth1}
\sum_{k\in\omega }  ||\Mb_{k}(x\uhr m)||_1<2 \text{ for all }m\in\omega.
\end{align}

\noindent {\bf Construction.}
\begin{enumerate}[1.]
\item On $2^{n_0}$, play the winning strategy
for the $(  c_0,d_0,n_0)$-\sided-game \ref{singdef2}
to control $\Mb_0$ (namely $||\Mb_0(\h\sigma)||_1\leq d_0$
for all $ \h\sigma\preceq\sigma^*$). The set $V_0$ consists of the strings enumerated during
this game.
\item When the strategy tells us to enumerate $\sigma$, take $\sigma$ as if it is
a \winningstring. Then, on $[\sigma]^\preceq\cap 2^{n_0+n_1}$,
play the winning strategy for the $(  c_1,d_1,n_1)$-\sided-game \ref{singdef2}
to control $\sum_{k<2} \Mb_k$
(namely $\sum_{k<2} ||\Mb_{k}(\h\sigma)||_1\leq d_1$
for all $\sigma\preceq\h\sigma\preceq\sigma^*$). The set $V_1$ consists of strings in $2^{n_0+ n_1}$
enumerated during the above game.
\item If at some point it is found that $\sigma\in 2^{n_0}$ is not
a \winningstring, then go back to the game on $2^{n_0}$, obtain the next string $\t \sigma$ by the winning
strategy and proceed with $\t\sigma$ (instead of $\sigma$) as in step (2).
\item Similarly for $V_2,V_3,\cdots$.
\end{enumerate}
Note that if $\sigma_0^*$ is a \winningstring\ on $2^{n_0}$,
then it follows that $||\Mb_{0}(\sigma_0^*)||_1\leq d_0$.
Since $\delta_1$ is sufficiently small,
we have $$\sum_{k<2} ||\Mb_{k}(\sigma_0^*)||_1<c_1.$$
Thus the game   on $[\sigma_0^*]^\preceq\cap 2^{n_0+n_1}$
will produce a $\sigma_1^*\in [\sigma_0^*]^\preceq\cap 2^{n_0+n_1}$ such that
 $$\sum_{k<2 }  ||\Mb_{k}(\h\sigma)||_1\leq d_1$$
 for all $\sigma_0^*\preceq\h\sigma\preceq\sigma_1^*$.
 In other words, each $\sigma_k^*$ exist, verifying (\ref{singeqproofth1}).
On the other hand, it's clear that $\m(V_k) \leq 2^{-k}$.
\end{proof}

\subsection{Nesting}\label{singsec2}
First we notice a nesting property of
the $(c,n)$-\sided-game \ref{singdef2}.
\begin{claim}\label{singclaim0}
Suppose \alice\ has a winning strategy for $(c_j,n_j)$-\sided-game \ref{singdef2} such that
$\m(A)\leq \varepsilon_j$
for each $j<2$.
Then \alice\ has a winning strategy for $(c_0c_1,n_0+n_1)$-\sided-game \ref{singdef2}
such that $\m(A)\leq \varepsilon_0\varepsilon_1$.
\end{claim}
\begin{proof}
In order to win the
$(c_0c_1,n_0+n_1)$-\sided-game we nest $(c_1,n_1)$-strategy into the
$(c_0c_1, c_1,n_0)$-strategy
\footnote{
We call the winning strategy for $(c_j,n_j)$-\sided-game with
$\m(A)\leq \varepsilon_j$ as $(c_j,n_j)$-strategy.
The winning strategy for $(c_0c_1, c_1,n_0)$-\sided-game
(with $\m(A)\leq \varepsilon_0$) exists since by scaling,
$(c_0c_1, c_1,n_0)$-\sided-game is equivalent to $(c_0,n_0)$-\sided-game.}.
For each $\rho\in 2^n$
let $(c_1,n_1)^{\rho}$-\sided-game denote the shift of $(c_1,n_1)$-\sided-game by $\rho$, in the sense that
each string $\sigma$ in the  $(c_1,n_1)$-\sided-game is replaced by $\rho\ast \sigma$.
Let's see what happens
if \alice\ plays the  $(c_1,n_1)^{\rho}$-strategy
(during a $(c_0c_1,n_0+n_1)$-\sided-game).
\begin{itemize}
\item If for each $\sigma\in [\rho]^\preceq\cap 2^{n_0+n_1}$ \alice\ enumerates,
$\sigma$ admits a \catchingpoint\ $\h\sigma$ (by the round $\sigma$ is enumerated)
such that $\h\sigma\succeq\rho$,
then the $(c_1,n_1)^{\rho}$-strategy will forces
\footnote{In case of any doubt, here ``a strategy forces $\varphi$" does not have anything to do with
the forcing notion in logic. It simply means the strategy will ensure that at some round $t$, the state of the
game will satisfy $\varphi$.}
 that at some round $t$, $||\Mb[t](\rho)||_1\geq c_1$.
\item Otherwise, for some $\sigma$ \alice\ enumerates,
the \catchingpoint\ satisfies $\h\sigma\preceq\rho$.
\end{itemize}
Which ever is the case,  we have
\begin{align}
\label{singeqiter0}
&||\Mb[t](\h\rho)||_1\geq c_1\text{ for some $\h\rho\preceq\rho$ }\\ \nonumber
&\text{ while the cost satisfies }
\m(A_\rho)\leq \varepsilon_1\m(\rho).
\end{align}
The strategy for $(c_0c_1,n_0+n_1)$-\sided-game is to
play $(c_0 c_1, c_1, n_0)$-\sided-game (at level $n_0$), but each time an enumeration $\sigma\in 2^{n_0}$
is instructed by the strategy, \alice\ instead play and win the shifted  $(c_1,n_1)^{\sigma}$-\sided-game, which we may view as a
sub-game. Note that the outcome $||\Mb(\h \rho)||_1\geq c_1$
for some $\h\rho\preceq\rho$ of the sub-game $(c_1,n_1)^{\rho}$ that is won  (see (\ref{singeqiter0}))
corresponds to a valid response by   Baby  on the {\em fictitious} move $\rho$  in $(c_0 c_1, c_1, n_0)$-\sided-game by   Alice.
Since the strategy for $(c_0c_1,c_1,n_0)$-\sided-game that we simulate is winning, it follows that the nested strategy is a
winning strategy for $(c_0c_1,c_1,n_0)$-\sided-game.
If $A_0\subseteq 2^{n_0}$ is the set of strings corresponding to the ficticious moves by Alice in $(c_0c_1,c_1,n_0)$-\sided-game, then
by (\ref{singeqiter0}), the cost of the nested strategy is at most
$\varepsilon_1\cdot \mu(A_0)\leq  \varepsilon_0\varepsilon_1$.
\end{proof}
\begin{claim}\label{singclaim1}
Suppose for every $a>0$, there is a  $\delta(a)>0$ (depending on $a$)
such that for every $0<\delta<\delta(a)$, there is an $n\in\omega$ such that \alice\ has a winning strategy for
  $(1-\delta,n)$-\sided-game \ref{singdef2}
such that $\m(A)\leq 1-a\delta$.
Then for every $ c<1, \varepsilon>0$, there is an $  n\in\omega $ such that
\alice\ has a winning strategy for   $(c,n)$-\sided-game \ref{singdef2} such that
$\m(A)\leq  \varepsilon$.
\end{claim}\begin{proof}
Fix $ c<1$. Let $a,\delta>0$ be such that
for some $k\in\omega$
\begin{align}\nonumber
&(1-a\delta)^k \leq \varepsilon,
(1-\delta)^k\geq c\text{ and let $n\in\omega$ so that }\\ \nonumber
&\text{\alice\ has a winning strategy
for $(1-\delta,n)$-\sided-game \ref{singdef2} }
\\ \nonumber
&\text{such that  $\m(A)\leq 1-a\delta$.}
\end{align}
Now the conclusion follows by applying Claim \ref{singclaim0} $k$ times.
\end{proof}
By  Claim \ref{singclaim1},
to prove Lemma \ref{singlem41}, it suffices to prove the following.

\begin{lemma}\label{singlem42}
For every $a>0$, there is a  $\delta(a)>0$ (depending on $a$)
such that for every $0<\delta<\delta(a)$,
 there is an $n\in\omega$
such that \alice\ has a winning strategy for
$(1-\delta,n)$-\sided-game \ref{singdef2}
such that $\m(A)\leq 1-a\delta$.
\end{lemma}

\subsection{Dynamic goal}
\label{singsec10}
For technical reason,
we consider the following game
which   allows
 \alice\ to achieve a dynamic goal
 depending on the resource she cost
 (the more she cost the better goal she is supposed to achieve).
Let $a>0, n\in\omega$.

 \begin{definition}[Dynamic $(a,n)$-\sided-game]
 \label{singdef3}

At each round $t\in\omega$, \alice\ firstly enumerates a $\sigma\in 2^n$
(that has not been enumerated before);
then \baby\ presents $i$-\sided\ supermartingale
$M_i[t]$ such that the following hold:
\begin{itemize}
\item $||\Mb[t](\h\sigma)||_1\geq 1 $ for some $\h\sigma\preceq\sigma$;
\item $M_i[t]$ dominates $ M_i[t-1]$ on $2^{\leq n}$.
\end{itemize}
Let $A[t]$ denote the set of $\sigma$ \alice\ enumerated by round $t$.
\alice\ wins the game if for some $t$, $1-||\Mb[t](\emptyset)||_1
\leq \frac{1}{a}(1-\m(A[t]))$.
\end{definition}
In case of any confusion, we emphasize that
$1-||\Mb[t](\emptyset)||_1\leq 0, \m(A[t])<1$ satisfies
the winning criterion of game \ref{singdef3}.
Clearly, dynamic game \ref{singdef3} has the same rule as
game \ref{singdef2} but a different winning criterion.

\begin{claim}\label{singclaim2}
Let $a,\varepsilon>0, \h n\in\omega$. Suppose
\alice\ has a winning strategy for the dynamic
$(a,\h n)$-\sided-game \ref{singdef3}
such that $\m(A)\leq 1-\varepsilon$.
Then for every $0<\delta\leq \varepsilon/2a$,
there is an $n\in\omega$ such that
\alice\ has a winning strategy for the
$(1-2\delta,n)$-\sided-game \ref{singdef2} such that
$\m(A)\leq 1-a\delta$.
\end{claim}\begin{proof}
Let $\t n$ be sufficiently large, say $2^{-\t n}<a\delta$.
Rougly speaking, the winning strategy goes like this.
\alice\ play the winning strategy of the dynamic game \ref{singdef3}
over each $\rho\in 2^{\t n}$ (while monitoring the cost  $\m(\t A_\rho)$
over each $\rho$) until $\int_{\t A[\t t]}(1-\m(\t A_\rho|\rho))\approx a\delta$
(where $\t A[t]$ is the set of $\rho\in 2^{\t n}$ over which the strategy has been played by round $t$).
Then \alice\ simply enumerates $2^{\t n}\setminus \t A[\t t]$.
If for each $\rho\in \t A[t]$, the dynamic winning criterion is reached, then
$$1-||\Mb[t](\emptyset)||_1
\leq \int_{\rho\in \t A[\t t]}(1-|| \Mb[t](\rho)||_1)\leq\int_{\rho\in \t A[\t t]}\frac{1}{a}(1-\m(\t A_\rho|\rho))  \lessapprox \delta.$$

A small difficulty is that it's not necessary that
for each  $\rho\in \t A[t]$, the dynamic winning criterion is reached
since some string over $\rho$ \alice\ enumerates may admit \catchingpoint\
that is a prefix of $\rho$ (i.e., the action of \baby\ may not
be a valid action with respect to the sub-game over $\rho$).
To overcome this, note that
(like we argued in the proof of Claim \ref{singclaim0}),
for each $\rho\in 2^{\t n}$, \alice\ could play
the winning strategy   for the dynamic $(a,\h n)$-\sided-game
\ref{singdef3}
(on $[\rho]^\preceq\cap 2^{\t n+\h n}$)
so as to force
\begin{align}\label{singeqdyn2}
&\text{either $\rho$ admits a \catchingpoint, namely}\\ \nonumber
&\hspace{1cm}\text{$|| \Mb[t](\h\rho)||_1\geq 1$ for some $\h\rho\preceq\rho$;} \\ \label{singeqdyn3}
&\text{or $1-||\Mb[t](\rho)||_1\leq \frac{1}{a}(1-\m(\t A_\rho|\rho)) $,}
\end{align}
where $\t A_\rho$ is the set of strings \alice\ enumerated on the game over $\rho$.

We design the following strategy of \alice\ for
the $(1-2\delta,\t n+\h n)$-\sided-game \ref{singdef2}.
\begin{enumerate}[1.]
\item Until certain criterion (namely (\ref{singeq2})) is reached,
for each $\rho\in 2^{\t n}$ (in whatever order):
  play
the winning strategy for the dynamic $(a,\h n)$-\sided-game
\ref{singdef3}
(on $[\rho]^\preceq\cap 2^{\t n+\h n}$)
so as to force either (\ref{singeqdyn2}) or (\ref{singeqdyn3})
\footnote{For example, if a \catchingpoint\ of $\rho$ already exist
(i.e., (\ref{singeqdyn2}) is already true),
then \alice\ does nothing over $\rho$ and  $\t A_\rho=\emptyset$.}.
\item Let $\t A[t]$ denote the set of $\rho\in 2^{\t n}$ that the
above action has been done with respect to $\rho$ (by round $t$).
\alice\ do the above action to each $\rho\in 2^{\t n}$ until
at some round $\t t$,
\begin{align}\label{singeq2}
 \int_{\rho\in \t A[\t t]}\big(1-\m(\t A_\rho|\rho)\big)\geq a\delta
\end{align}
\item Now (right after the round $\t t $)
 \alice\ simply enumerates $ 2^{\t n}\setminus \t
A[\t t]$
(enumerate $B$ means enumerate every elements in $[B]^\preceq\cap 2^{\t n+\h n}$).
\end{enumerate}
Clearly, the round $\t t$ must exist since,
by the hypothesis of this claim,
$1-\m(\t A_\rho|\rho)\geq \varepsilon\geq a\delta$.
Since $\t n$ is sufficiently large and
during each sub-game, $\int_{\rho\in \t A[t]}(1-\m(\t A_\rho|\rho))$
increases at most $2^{-\t n}$, we may assume
\begin{align}\label{singeq1}
\int_{\rho\in \t A[\t t]}\big(1-\m(\t A_\rho|\rho)\big)\leq 2a\delta.
\end{align}
By observation (\ref{singeqdyn2}), (\ref{singeqdyn3})
(and the nondecreasing, supermartingale of $M_i[t]$,
let $t$ be the round
everything is done),
   actions (1), (3) force that
 \begin{align}\nonumber
\bullet\ \  &\text{for every }\rho\in 2^{\t n}\setminus \t A[\t t],\ (\ref{singeqdyn2}) \text{ follows;}\\ \nonumber
 &\text{for every }\rho\in \t A[\t t],\text{ either (\ref{singeqdyn2}) or (\ref{singeqdyn3}) follows.}
 \end{align}
 Thus, there must be a \pf
set $B\subseteq 2^{\leq \t n}$ with $[B]^\preceq\supseteq 2^{\t n}\setminus \t A[\t t]$
such that
\begin{align}\label{singeqdyn0}
1-||\Mb[t](\rho)||_1&\leq \frac{1}{a}(1-\m(\t A_\rho|\rho))  &\text{ for all }\rho\in \t A[\t t]\setminus [B]^\preceq;
\\ \label{singeqdyn1}
|| \Mb[t](\h\rho)||_1&\geq 1 &\text{ for all }\h\rho\in B.
\end{align}
Thus, by supermartingale of $M_i[t]$,
\begin{align}\nonumber
1-||\Mb[t ](\emptyset)||_1
&\overset{(\ref{singeqdyn1})}{\leq}  \int_{\rho\in \t A[\t t]\setminus [B]^\preceq}\bigg(1-||\Mb[t](\rho)||_1\bigg)
\\ \nonumber
& \overset{(\ref{singeqdyn0})}{\leq}
\int_{\rho\in \t A[\t t]\setminus [B]^\preceq}\frac{1}{a}\big(1-\m(\t A_\rho|\rho)\big)
 \\ \nonumber
 &\overset{(\ref{singeq1})}
{\leq}  2\delta.
\end{align}
On the other hand, the cost of \alice, namely $\m(A)$, satisfies:
\begin{align}\nonumber
1-\m(A) = \int_{\rho\in\t A[\t t]}\big(1-\m(\t A_\rho|\rho)\big)
\overset{(\ref{singeq2})}{\geq} a\delta.
\end{align}
Thus we are done.
\end{proof}
\begin{remark}\label{singrem1}
The proof of Claim \ref{singclaim2}
does not take advantage of $i$-\sided ness
but the mere fact that $M_i[t]$ is a supermartingale
nondecreasing with respect to $t$.
\end{remark}
By Claim \ref{singclaim2}, Lemma \ref{singlem42} is reduced to the following assertion:
\begin{align}\nonumber
\bullet\ \ &\text{For every $a>0$, there exists a $\varepsilon>0$ and $n\in\omega$
such that \alice\ has }\\ \nonumber
&\text{a winning strategy for the dynamic $(a,n)$-\sided-game \ref{singdef3}
such that }\\ \nonumber
&\text{$\m(A)\leq 1-\varepsilon$.}
\end{align}
Combine with the upward closure
of game \ref{singdef3}
(i.e., if \alice\ has a winning strategy for the dynamic
$(a,n)$-\sided-game \ref{singdef3} such that
$\m(A)\leq 1-\varepsilon$, then she also has a winning strategy for
the dynamic $(a,\h n)$-\sided-game \ref{singdef3}
such that $\m(A)\leq 1-\varepsilon$ provided $\h n\geq n$),
   Lemma \ref{singlem42} is reduced to
the following:
\begin{lemma}\label{singques0}
For every $a>0$,  there is an $n\in\omega$ such that
\alice\ has a winning strategy for the dynamic $(a,n)$-\sided-game \ref{singdef3}
such that $\m(A)<1$.
\end{lemma}

\subsubsection{The power of dynamic goal —— \typeaattention}
\label{singsec11}
The dynamic goal argument is arguably the most important ingredients
of the proof of Lemma \ref{singlem41}.
Let's briefly explain the power of dynamic goal (i.e., why it is
considerably easier to win the dynamic  game \ref{singdef3}).
During the dynamic $(a,n)$-\sided-game \ref{singdef3} (or some similar game),
\begin{definition}[\typeaattention]
We say $\rho\in 2^{\leq n}$ \emph{receives \typeaattention}
at round $t$
iff:
\begin{align}\label{singeq11}
1-||\Mb[t](\rho)||_1&\leq\frac{1}{a} (1-\m(A[t]|\rho))\text{ \ while }\\ \nonumber
\m(A[t]|\rho)&<1.
\end{align}
\end{definition}
The point is,
\begin{fact}\label{singfactwinatt}
If some $\rho$ receives \typeaattention, then \alice\ wins
the dynamic $(a,n)$-\sided-game
\ref{singdef3} (with $\m(A)<1$)
 by enumerating (whatever is left in)
  $2^n\setminus [\rho]^\preceq$.
  \end{fact}
  \begin{proof}
  To see this, upon doing so,
  it forces
  every $\t\rho\in 2^{|\rho|}\setminus \{\rho\}$ to admit a \catchingpoint.
    Therefore, we have, by supermartingale of $M_i[t]$, (\ref{singeq11})
    and $A[t+1]=A[t]\cup( 2^n\setminus[\rho]^\preceq)$,
  \begin{align}\nonumber
1-||\Mb[t+1](\emptyset)||_1&\leq
(1-||\Mb[t+1](\rho)||_1)\m(\rho)\\ \nonumber
&\leq \frac{1-\m(A[t]|\rho)}{a}\m(\rho)\\ \nonumber
&= \frac{1}{a}(1-\m(A[t+1])).
\end{align}
On the other hand, $A[t+1]\cap [\rho]^\preceq = A[t]\cap [\rho]^\preceq\ne\emptyset$.
Thus $\m(A[t+1])<1$.
\end{proof}
Note: whether $\rho$ receives \typeaattention\
only relies on what happens in $[\rho]^\preceq$.

\subsection{Imposing   restrictions on the \sided\ player}\label{singsec1}
The dynamic winning criterion allows us to impose very strong restriction on
the \sided\ player's action.
\begin{definition}[Restricted dynamic $(a,\delta,n)$-\sided-game]
 \label{singdef4}
At each round $t\in\omega$, \alice\ firstly enumerates a $\sigma\in 2^n$
(that has not been enumerated before);
then \baby\ presents $i$-\sided\
 supermartingale
$M_i[t]$ such that the following hold:
\begin{itemize}
\item $||\Mb[t](\sigma)||_1\geq 1 $;
\item $||\Mb[t](\rho)||_1\leq 1+\delta$
for all $\rho\in 2^{\leq n}$;
\item $M_i[t]$ dominates $ M_i[t-1]$ on $2^{\leq n}$
(we set $M_{i}[-1]\equiv 0$).
\end{itemize}
Let $A[t]$ denote the set of $\sigma$ \alice\ enumerated by round $t$.
\alice\ wins the game if for some $t$, $1-||\Mb[t](\emptyset)||_1
\leq \frac{1}{a}(1-\m(A[t]))$.
\end{definition}
Clearly the restricted dynamic game \ref{singdef4} has
the same winning criterion as dynamic game \ref{singdef3} but different
rules.
We refer to the additional requirement
$||\Mb[t](\sigma)||_1\geq 1$
(instead of $||\Mb[t](\h\sigma)||_1\geq 1$ for some $\h\sigma\preceq\sigma$)
and $||\Mb[t](\rho)||_1\leq 1+\delta$
as \emph{restriction rule I,
restriction rule II} respectively.
Let's firstly focus on the restriction rule II.
The intuition   (how a winning strategy for the restricted game
helps her win the non restricted game) is the following.
If, in the non restricted dynamic game,
 at some round $t$, \baby\ is forced to set $|| \Mb[t](\rho)||_1>1$,
then it is likely that $\rho$ receives \typeaattention\ (at  round $t$)
since we now have $1-||\Mb[t](\rho)||_1<0$. So upon this (namely $|| \Mb[t](\rho)||_1>1$)
happening,
\begin{align}\label{singeq12}
&\text{as long as some string
in $[\rho]^\preceq\cap 2^n$ is not enumerated by round $t$,}\\ \nonumber
&\text{$\rho$ receives \typeaattention.}
\end{align}
 How do \alice\ ensure (\ref{singeq12}).
 That is simple, whenever she wants to enumerate $\sigma$, instead of enumerating
 $\sigma$, she enumerates $([\sigma]^\preceq\cap 2^{n+\h n})\setminus \{\sigma 0^{\h n}\}$.
 i.e., she reserves the string $\sigma 0^{\h n}$
 to prepare for the \typeaattention: $1-|| \Mb[t](\sigma)||_1<0$.
 By doing so, she could no longer  force $|| \Mb[t](\h\sigma)||_1\geq 1$ (for some $\h\sigma\preceq\sigma$),
 but $|| \Mb[t](\h\sigma)||_1\geq 1-2^{-\h n}$ (for some $\h\sigma\preceq\sigma$).
 Clearly, letting $\h n$ to be sufficiently large,
 this small difference can be ignored.
The intuition of restriction rule I is similar.
The concrete proof goes as follows.

\begin{claim}\label{singclaim4}
Let $a,\delta>0, \t n\in \omega$. Suppose \alice\ has a winning strategy for
the restricted dynamic $(a, \delta, \t n)$-\sided-game \ref{singdef4}
such that $\m(A)<1$.
Then there is an $n\in\omega$ such that \alice\ has a winning
 strategy for
the  dynamic $(a/2, n)$-\sided-game \ref{singdef3}
such that $\m(A)<1$.
\end{claim}
\begin{proof}
Let $\h n$ be sufficiently large so that
$\Delta = 2^{-\h n}$ is sufficiently small (depending on $a,\delta,\t n$ and to be specified).
Firstly, we note that for $\rho\in 2^{\t n}$, upon enumerating
$D_\rho:=([\rho]^\preceq\cap 2^{\t n+\h n})\setminus \{\rho0^{\h n}\}$,
\alice\ forces (when playing the dynamic \sided-game \ref{singdef3}) that
\begin{align}\label{singeqres1}
&\text{either  $\rho$ admits a \catchingpoint. i.e., }
\\ \nonumber
&\hspace{1cm}\text{$||\Mb[t](\h\rho)||_1\geq 1$ for some $\h\rho\preceq\rho$;} \\ \label{singeqres0}
&\text{or $1-||\Mb[t](\rho)||_1\geq 1-\Delta$.}
\end{align}
A winning strategy for   dynamic $(a/2, \t n+\h n)$-\sided-game \ref{singdef3}
proceeds as follows:
\begin{enumerate}[1.]
\item On $2^{\t n}$, \alice\ plays the  winning strategy for the
($(1-\Delta)$-scaled
\footnote{In a $(1-\Delta)$-scaled restricted dynamic $(a,\delta,\t n)$-\sided-game,
the restriction rule I, II are
$||\Mb[t](\sigma)||_1\geq 1-\Delta $,
 $||\Mb[t](\rho)||_1\leq (1-\Delta)(1+\delta)$ respectively.
 The winning criterion is
 $(1-\Delta)-||\Mb[t](\emptyset)||_1
\leq (1-\Delta)\cdot\frac{1}{a}(1-\m(A[t]))$.
 })
restricted dynamic $(a, \delta,\t n)$-\sided-game \ref{singdef4}
until  some string $\rho\in 2^{\leq \t n+\h n}$ receives \typeaattention\
(with respect to the dynamic $(a/2, \t n+\h n)$-\sided-game \ref{singdef3});
\item But when the strategy tells her to enumerate $\rho\in 2^{\t n}$,
instead of   enumerating $\rho$, she enumerates
$D_\rho$,
which forces either (\ref{singeqres1}) or (\ref{singeqres0}).
\item When some $\rho\in 2^{\leq \t n+\h n}$ receives \typeaattention,
\alice\ then enumerates $2^{\t n+\h n}\setminus [\rho]^\preceq$.
\end{enumerate}
Clearly, there are three cases depending
on whether \baby\ breaks the restriction rule I, II
of the ($(1-\Delta)$-scaled) restricted dynamic game \ref{singdef4} played on $2^{\t n}$.

\ \\

\noindent\textbf{Case 1:} \baby\ breaks the restriction rule I.
i.e.,  for some $\rho\in 2^{\t n}$, where $D_\rho$
is enumerated (at round $t$),
(\ref{singeqres0}) does not follow.

As we observed,
this means (\ref{singeqres1}) follows. But if
 \catchingpoint\ $\h\rho$ of $\rho\in 2^{\t n}$ appears,
then $\h\rho$ clearly receives \typeaattention\ (with respect to the dynamic $(a/2,\t n+ \h n)$-\sided-game \ref{singdef3}),
since $1-||\Mb[t](\h\rho)||_1\leq 0$ while $\m(A[t]|\h\rho)<1$
(since $\rho 0^{\h n}\notin A[t]$).
As we argued in Fact \ref{singfactwinatt}, action (3) makes
\alice\ win the dynamic $(a/2,\t n+\h n)$-\sided-game \ref{singdef3}
(with $\m(A)<1$).

\ \\

\noindent\textbf{Case 2: }\baby\ breaks
the restriction rule II. i.e., (since $\Delta$ is sufficiently small depending on $\delta$)
 for some $\rho\in 2^{\leq \t n}$
\begin{align}\nonumber
||\Mb[t](\rho)||_1&> (1-\Delta)(1+\delta)\\ \nonumber
&\geq 1+\delta/2.
\end{align}

However, that implies $\rho$ receives \typeaattention\
  since $1-||\Mb[t](\rho)||_1<0$
while $\m(A[t]|\rho)<1$.
By Fact \ref{singfactwinatt}, action (3) makes
\alice\ win the dynamic $(a/2,\t n+\h n)$-\sided-game \ref{singdef3}
(with $\m(A)<1$).

\ \\

\noindent\textbf{Case 3: }\baby\ does not break the restriction rules.

Let $\t A[t]$ denote the set of $\rho\in 2^{
\t n}$ such that $D_\rho$ is enumerated by round $t$.
The winning strategy (of the $(1-\Delta)$-scaled restricted dynamic $(a,\t n)$-\sided-game \ref{singdef4} on $2^{\t n}$)
  forces that at some round $t$,
\begin{align}\nonumber
(1-\Delta)- ||\Mb[t](\emptyset)||_1&\leq (1-\Delta)\cdot \frac{1}{a}(1-\m(\t A[t])).
\end{align}
Then we have (since $\Delta$ is sufficiently small depending on $a,\t n$,
so $\frac{1}{a}(1-\m(\t A[t]))\geq 2^{-\t n}/a\geq \Delta$)
\begin{align}\nonumber
1- ||\Mb[t](\emptyset)||_1&\leq (1-\Delta)\cdot\frac{1}{a}(1-\m(\t A[t]))+\Delta
\\ \nonumber
&\leq \frac{2}{a}(1-\m(  A[t]))\text{ while }\\ \nonumber
\m(A[t])&  <1
\end{align}
So \alice\ wins the dynamic $(a/2,\t n+\h n)$-\sided-game
with $\m(A) <1$.
Thus we are all done.
\end{proof}
\begin{remark}\label{singrem2}
The proof of Claim \ref{singclaim4}
does not take advantage of $i$-\sided ness
but the mere fact that $M_i[t]$ is a supermartingale
nondecreasing with respect to $t$.
\end{remark}
By Claim \ref{singclaim4},
to prove Lemma \ref{singques0},
it suffices to prove the following.
\begin{lemma}
\label{singlem43}

For every $a>0$,
there are $ n\in\omega,\delta>0$ such that
\alice\ has a winning strategy for   restricted
dynamic $(a,\delta, n)$-\sided-game \ref{singdef4}
such that $\m(A)<1$.
\end{lemma}

\subsection{Proof of Lemma \ref{singlem43}}
\label{singsec0}

Fix $a>0$,
 a sufficiently small $\delta>0$ (depending on $a$) and
a sufficiently large $n$ (depending on $a $).
We start with the following observation (\ref{singobs01}).
For a nonempty
\pf  set $B\subseteq 2^{<\omega}$,
a function $\Mb:2^{<\omega}\rightarrow \mathbb{R}^k$
(not necessarily a vector-martingale),
we write
\begin{align}\nonumber
&\mbE(\Mb|B) = \frac{1}{\m(B)}\int_B \Mb\text{ and }\\ \nonumber
&\variance(\Mb|B)= \frac{1}{\m(B)}\int_{\sigma\in B} ||\Mb(\sigma) - \mbE(\Mb|B) ||_2^2.
\end{align}
In   \S  \ref{singsec0}, for a $\rho\in 2^{<\omega}$, let $B_\rho = \{\rho11,\rho0\}, \tau_\rho = \rho10$.
The point is,
\begin{align}\label{singobs01}
&\text{if $B_\rho$ is enumerated
while nothing in $[\tau_\rho]^\preceq$ is enumerated, then }
\\ \nonumber
 &\text{ either $\rho$ receives \typeaattention;}\\ \nonumber
 &\text{ or \baby\ is forced to allocate a constant large   variance on $B_\rho$.}
 \end{align}

\begin{claim}\label{singclaim83}
Let $M_i$ be $i$-\sided\ supermartingale
(not necessarily martingale).
Suppose $||\Mb(11)||_1$,$||\Mb(0)||_1\geq 1$.
Then, let $B = \{0,11\}$,
we have $1-||\Mb(\emptyset)||_1\leq
C\sqrt{\variance(\Mb|B)}$
where $C$ is an absolute constant.

\end{claim}
\begin{proof}
This is because  if the variance is small, then
$\min\{ M_i(0),M_i(11)\}$
is close to $\mbE( M_i|B)$;
and by $i$-\sided\ of $M_i$,
$  M_i(\emptyset)\geq \min\{ M_i(0),M_i(11)\}$.
More specifically,
$\variance(M_i|B)\geq \frac{1}{3}\big(\min\{M_i(0),M_i(11)\}- \mbE(M_i|B)\big)^2.$
So,
suppose $\sqrt{\variance(\Mb|B)} = \Delta$.
We have: for some absolute constant $C_0$ (say $C_0=2$),
$$\mbE(M_i|B)-\min\{M_i(0),M_i(11)\} \leq C_0\Delta.$$
By $i$-\sided\ of $M_i$,
$M_i(\emptyset)\geq \min\{M_i(0),M_i(11)\}$.
 Thus
\begin{align}\nonumber
1-||\Mb(\emptyset)||_1 &\leq  \mbE(||\Mb||_1|B)
 -\sum\nolimits_{i}\min \{M_i(0),M_i(11)\}
 \\ \nonumber
 &\leq 2C_0\Delta.
 \end{align}
\end{proof}

The strategy of \alice\ is to enumerate $\sigma\in 2^n$ in some fixed order so that
for many $\rho\in 2^{\leq n}$, there is a round $t$ such that
$B_\rho$ has been enumerated (by round $t$)  while nothing in $[\tau_\rho]^\preceq$
 is enumerated (by round $t$).
 This forces \baby\ to allocate  constant large variance on $B_\rho$ for many $\rho$ (if no $\rho$ receives \typeaattention),
 which contradicts  to the variance analysis
  (Claim \ref{singclaim5}).

Let $\t B_0,\cdots,\t B_{m}\subseteq 2^{\leq n}$ be
a sequence of \pf  sets such that
$[\t B_{\t m}] = 2^\omega, \t B_{\t m}\subseteq [\t B_{\t m-1}]^\preceq$.
For each $\rho\in \t B_{\t m-1}$, let $\t B_\rho = [\rho]^\preceq\cap \t B_{\t m}$.
Let $\varepsilon>0$.
\begin{claim}
\label{singclaim5}
Let $(M_j,j<k)$ be supermartingales (not necessarily single-\sided)
such that $1-\varepsilon\leq \int_{\sigma\in 2^n}||\Mb(\sigma)||_1$,
 $||\Mb(\emptyset)||_1\leq 1$ and
$ M_j(\rho)\leq 2$ for all $\rho\in 2^{\leq n}$.
Then, for some constant $C(k)$ (depending on $k$),
$$\sum_{\t m<m}\int\nolimits_{\rho\in \t B_{\t m}}\variance(\Mb|\t B_\rho)
\leq C(k)(1+m\varepsilon).$$

\end{claim}
\begin{proof}
Let $\h M_j$ be a martingale
generated by $M_j\uhr 2^n$ (i.e., $\h M_j$ agrees with
$M_j $ on $2^n$).
Note that $M_j$ dominates $\h M_j$ on $2^{\leq n}$.
Let $\Delta M_j =M_j-\h M_j $ and
$\Delta\Mb = (\Delta M_0,\cdots,\Delta M_{k-1})$.
Note that (by $(z+y)^2\leq 2z^2+2y^2$) for any \pf  set $B$,
\begin{align}\label{singeq6}
\variance(\Mb|B)\leq 2\variance(\h \Mb|B)
+2\variance(\Delta\Mb|B).
\end{align}
We need to estimate each term of (\ref{singeq6}).

Concerning the first term,
recall that
the variance of a random variable equals to the expectation of
its conditional variance plus the variance of its expectation.
More specifically, let $z$ be a random variable, $\mcal{F}$ be a filter
(say a finite collection of mutually disjoint events whose union is the whole probability space),
let $y = \mbE(z|\mcal{F})$; then $ \mbE(\variance(z|\mcal{F}))=\variance(z)-\variance(y)$.
Using this\footnote{Where we take $\t B_{\t m+1}$ as
the probability space, $\h \Mb\uhr \t B_{\t m+1}$ as the random variable $z$, $\{\t B_\rho:\rho\in \t B_{\t m}\}$
as the filter.} and by martingale of $\h M_j$, we have
$$
\int\nolimits_{\rho\in \t B_{\t m}}
\variance(\h \Mb|\t B_\rho) = \variance(\h\Mb|\t B_{\t m+1})- \variance(\h\Mb|\t B_{\t m}).
$$
Therefore, since $range(\h M_j)\subseteq [0,2]$, for some constant $C_0(k)$ depending on $k$ (say $C_0(k) = 4k$),
\begin{align}\label{singeq20}
\sum_{\t m<m}\int\nolimits_{\rho\in \t B_{\t m}}
\variance(\h\Mb|\t B_\rho)\leq C_0(k).
\end{align}

Concerning the second term, since $range(\Delta M_j)\subseteq [0,2]$,  we have
for every \pf  set $B$:
\begin{align}\label{singeq8}
\variance(\Delta\Mb|B)\leq \mbE\big(||\Delta \Mb||_2^2 \big|B\big)\leq
2\mbE\big(||\Delta \Mb ||_1 \big |B\big).
\end{align}
By supermartingale of $\Delta M_j$,
\begin{align}\label{singeq7}
\int_{\rho\in \t B_{\t m-1}}\mbE\big(||\Delta \Mb||_1\big|\t B_\rho\big)
 =& \mbE\big(||\Delta \Mb||_1\big|\t B_{\t m}\big)
 \\ \nonumber
 \leq& ||\Delta \Mb(\emptyset)||_1\leq \varepsilon.
 \end{align}
Combining (\ref{singeq6}), (\ref{singeq20}), (\ref{singeq8}), (\ref{singeq7}) concludes the proof.
\end{proof}

\alice\ will enumerate strings in $2^n$ by a specific fixed order
so that for many $\rho\in 2^{\leq n}$, there is a round $t$ such that
$B_\rho$ has been enumerated (by round $t$)  while nothing in $[\tau_\rho]^\preceq$
 is enumerated (by round $t$).
 To define this order, we fix an embedding $e:3^{\leq n/2}\rightarrow 2^{\leq n}$
inductively defined as follows:
\begin{align}
\nonumber
&e(\emptyset) =\emptyset;\text{ and
suppose $e(\alpha) = \rho$, then }\\ \nonumber
&e(\alpha 0 ) = \rho0, e(\alpha 1) = \rho 11,
e(\alpha 2) = \rho10.
\end{align}
Let $\leq_{lex}$ denote the lexicographical
order on $3^{n/2}$
(for example $ 01<_{lex}02<_{lex}10$).

 \begin{strategyn*}
\alice\ enumerates $(e(\alpha):\alpha\in 3^{n/2})$ in the lexicographical order on $3^{n/2}$
until some   $\rho\in 2^{\leq n}$ \emph{receives \typeaattention}.
When that happens, \alice\ then enumerates $2^n\setminus[\rho]^\preceq$.
 (End of   Strategy.)
\end{strategyn*}
Clearly there are two cases.

\noindent\textbf{Case (a)}: Some $\rho$ receives \typeaattention\ during the game.

As we argued in Fact \ref{singfactwinatt}, this means \alice\ win the game with $\m(A)<1$.

\noindent\textbf{Case (b)}: Otherwise.

We will derive a contradiction.
The strings \alice\ expects to receive \typeaattention\ are the elements
of $e(3^{\leq n/2})$.
Now since they don't, each of them will force \baby\ to allocate
a constant large \variancee.
To see this, for each $\rho\in e(3^{\leq n/2})$,
by the lexicographical order,
there is a round $t$ such that
\begin{align}\label{singeq14}
B_\rho\text{ is enumerated  while nothing in }[\tau_\rho]^\preceq
\text{ is enumerated.}
\end{align}
We refers to (\ref{singeq14}) as $\rho$ \emph{receives \potentialwin\ attention}.
Clearly
\begin{align}\label{singeq00}
1-\m(A[t]|\rho)= \m(\tau_\rho|\rho)=2^{-2}.
\end{align}
Now since $\rho$ does not receive \typeaattention
\footnote{This, namely (\ref{singeq00}), (\ref{singeq01}), is why it is crucial that nothing in $[\tau_\rho]^\preceq$
is enumerated when $B_\rho$ has been enumerated. And that's why we need to
enumerate in the lexicographical order.}
\begin{align}\label{singeq01}
1-||\Mb[t](\rho)||_1\geq 2^{-2}/a.
\end{align}
By Claim \ref{singclaim83}, for some constant $C(a)$
(say $C(a) = \frac{1}{2^4 C^2 a^2}$ where
$C$ is the constant in Claim \ref{singclaim83})
depending on $a$,
$$\variance(\Mb[t]|B_\rho)\geq C(a).$$
But $||\Mb[t](\t\tau)||_1\geq 1$  for each $\t\tau\in B_\rho$,
which means by the restriction rule II
(game  \ref{singdef4} second item),
$\variance(\Mb[\h t]|B_\rho)\geq C(a)-4\delta$\text{ for all }$\h t\geq t.$
Let
$\t B_\rho = B_\rho\cup\{\tau_\rho\}$,
since $\m(B_\rho)\geq \m(\t B_\rho)/2$ and $\delta$ is sufficiently small
(depending on $a$), we have
\begin{align}\label{singeq15}
\variance(\Mb[\h t]|\t B_\rho)\geq C(a)/4\text{ for all }\h t\geq t.
\end{align}
Let $\t B_m = e(3^m)$
(which verifies the setting of $\t B_m$
in Claim \ref{singclaim5})
 we have: for some round $t$ by the near end of the game,
 \begin{align}
 \label{singeq16}
 &\text{sufficiently many $\rho$ in $e(3^{\leq n/2})$,
  say every $\rho\in \cup_{m< n/4}e(3^{m})$,  }\\ \nonumber
  &\text{have received }
 \text{\potentialwin\ attention
 (see (\ref{singeq14}));
 and}\\ \label{singeq17}
 &\text{\alice\ has enumerated sufficiently many strings}.
 \end{align}
i.e.,
\begin{align}\label{singeq13}
\sum_{m<n/2-1}\int_{\rho\in \t B_m}\variance(\Mb[t]|\t B_\rho)
&\overset{(\ref{singeq15}),(\ref{singeq16})}{\geq} \frac{C(a)n}{16}  \text{ while }\\ \label{singeq18}
\int_{\sigma\in 2^n}||\Mb[t](\sigma)||_1&\overset{(\ref{singeq17})}{\geq} 1-\varepsilon.
\end{align}
Where   $\varepsilon$ is sufficiently small (say $o(C(a))$,
since it can  be chosen to be $2^{-n}$ and $n$ is sufficiently large depending on $a$).
Since $\emptyset$ does not receive \typeaattention, so
\begin{align}\label{singeq19}
||\Mb[t](\emptyset)||_1\leq 1.
\end{align}
But (\ref{singeq18}), (\ref{singeq19})
together with  Claim \ref{singclaim5} implies (for some absolute constant $C$)
\begin{align}\nonumber
\sum_{m<n/2-1}\int_{\rho\in \t B_m}\variance(\Mb[t]|\t B_\rho)
\leq C(1+n\varepsilon).
\end{align}
Thus a contradiction to (\ref{singeq13}).
By the choice of $C(a)$, we can choose  $n=O(a^2)$.

%% file: partialsided5.tex
\section{Defeat   kastergales}
\label{singsecpartial}


Again, Theorem \ref{singth1} is reduced to the winning strategy of the following game.
Let $0\leq c\leq 1,n\in\omega$.
We will use $p$ to denote partial function
on $2^{<\omega}$ with $range(p)\subseteq 2$;
we say function $p_1$ \emph{extends} $p_0$
if $dom(p_0)\subseteq dom(p_1)$
and $p_1$ agrees with $p_0$ on $dom(p_0)$.

 \begin{definition}[$(c,n,k)$-partial-\sided-game]
 \label{singdef962}

At each round $t\in\omega$, \alice\ firstly enumerates a $\sigma\in 2^n$
(that has not been enumerated before);
then \baby\ presents a partial function $p_j[t]$ (for each $j<k$),
 $p_j[t]$-\sided\ supermartingale
$M_j[t]$ (for each $j<k$) such that the following hold:
\begin{itemize}
\item $p_j[t]$ extends $p_j[t-1]$
(we set $dom(p_j[-1])=\emptyset$);
\item $|| \Mb[t](\h\sigma)||_1\geq 1 $ for some $\h\sigma\preceq\sigma$;
\item $M_j[t]$ dominates $ M_j[t-1]$ on $2^{\leq n}$.

\end{itemize}
\alice\ wins the game if for some $t$, $||\Mb[t](\emptyset)||_1\geq  c$.
\end{definition}

Until \S  \ref{singsec100},
fix a $k\in\omega$ (unless claimed otherwise).

\begin{lemma}\label{singlem941}
For
every  $ c<1, \varepsilon>0$, there is an $n\in\omega$ such that
\alice\ has  a winning strategy   for
$(c,n,k )$-partial-\sided-game \ref{singdef962} such that $\m(A)\leq \varepsilon$.
\end{lemma}

Again, we reduce (by nesting,
dynamic goal and restricting \sided\ player's action)
partial-\sided-game \ref{singdef962} to the corresponding
restricted dynamic partial-\sided-game \ref{singdef964}
(and Lemma \ref{singlem941} to Lemma \ref{singlem943}).
This part, \S \ref{singsecidr2}, is the same as
\S \ref{singsec2}-\S \ref{singsec1}.
 The new ingredients are the following.

\begin{enumerate}[(1)]
\item The idea
of Lemma \ref{singlem43} is to force \baby\ to allocate
a constant large variance on $2^{|\rho|+2}\cap [\rho]^\preceq$ for many
$\rho\in 2^{\leq n}$.  This is done by enumerating $B_\rho$ while
  nothing in $[\tau_\rho]^\preceq$ is enumerated.
For kastergale, there isn't such a simple strategy
(to force \baby\ to allocate variance on $2^{|\rho|+m}\cap [\rho]^\preceq$).
However, assuming a strategy to force that,
we show that \alice\ can win the restricted dynamic partial-\sided-game \ref{singdef964}
(so Lemma \ref{singlem943}--a winning strategy for
restricted dynamic game \ref{singdef962}, is reduced to Lemma \ref{singlem966}--a
 winning strategy for \variancee\ game \ref{singdef965}).

\item
An inductive (in $k$) proof that \alice\ has a winning strategy for   \variancee\ game \ref{singdef965}.
\end{enumerate}
The following diagram illustrates the framework of the new ingredients.

\begin{align}\nonumber
 &\text{Lemma \ref{singlem943}}
 \overset{Claim \ref{singclaim966}}{\hookrightarrow }
 \text{Lemma \ref{singlem966}
  (reduce to \variancee\ game \ref{singdef965}, \S  \ref{singseccatching})};
  \\ \nonumber
  &\text{A winning strategy for the \variancee\ game  \ref{singdef965}
  (proof of Lemma \ref{singlem966}, \S  \ref{singsec100})}.
\end{align}

\begin{proof}[Proof of Theorem \ref{singth1} using Lemma \ref{singlem941}]
The same as how Theorem \ref{singth3} is proved using Lemma \ref{singlem41}.
\end{proof}

Notably, other than a small portion
of the proof of Lemma \ref{singlem966}
(that \alice\ has a winning strategy for the \variancee\ game
\ref{singdef965} when $k=1$),
none of the ingredients    relies on partial-\sided ness
but the nondecreasing and supermartingale of $M_j[t]$.

\subsection{Nesting, dynamic goal and restrictions}
\label{singsecidr2}
We first note a nesting property of
the $(c,n,k )$-partial-\sided-game \ref{singdef962}.
By the same argument we used for Claim \ref{singclaim0} we have:
\begin{claim}\label{singclaim960}
Suppose
\alice\ has a winning strategy for the $(c_j,n_j,k )$-partial-\sided-game \ref{singdef962} such that
$\m(A)\leq \varepsilon_j$
(for each $j<2$).
Then \alice\ has a winning strategy for $(c_0c_1,n_0+n_1,k)$-partial-\sided-game \ref{singdef962}
such that $\m(A)\leq \varepsilon_0\varepsilon_1$.
\end{claim}
By the argument we used for Claim \ref{singclaim1} we have:
\begin{claim}\label{singclaim961}
Suppose for every $a>0$, there is a $\delta(a)>0$
(depending on $a$) such that for every $0<\delta<\delta(a)$,
 there is an  $n\in\omega$ such that
 \alice\ has a winning strategy for
the $(1-\delta,n,k)$-partial-\sided-game \ref{singdef962}
such that $\m(A)\leq 1-a\delta$.
Then for every $  c<1, \varepsilon>0$, there is an $n\in\omega$ such that
\alice\ has a winning strategy for   $(c,n,k)$-partial-\sided-game \ref{singdef962} such that
$\m(A)\leq  \varepsilon$.
\end{claim}
By  Claim \ref{singclaim961},
to prove Lemma \ref{singlem941}, it suffices to prove the following.

\begin{lemma}\label{singlem942}
For every $a>0$, there is a $\delta(a)>0$
(depending on $a$) such that for every $0<\delta<\delta(a)$,
 there is an $n\in\omega$ such that
 \alice\ has a winning strategy for
$(1-\delta,n)$-partial-\sided-game \ref{singdef962}
such that $\m(A)\leq 1-a\delta$.
\end{lemma}
Again, we reduce the $(c,n)$-partial-\sided-game \ref{singdef962}
to a dynamic game.
 \begin{definition}[Dynamic $(a,n, k )$-partial-\sided-game]
 \label{singdef963}
Let $a>0, n\in\omega$.
At each round $t\in\omega$, \alice\ firstly enumerates a $\sigma\in 2^n$
(that has not been enumerated before);
then \baby\ presents $p_j[t]$,
 $p_j[t]$-\sided\ supermartingale
$M_j[t]$ such that the following hold:
\begin{itemize}
\item $p_j[t]$ extends $p_j[t-1]$;
\item $||\Mb[t](\h\sigma)||_1\geq 1 $ for some $\h\sigma\preceq\sigma$;
\item $M_j[t]$ dominates $ M_j[t-1]$ on $2^{\leq n}$.
\end{itemize}
Let $A[t]$ denote the set of $\sigma$ \alice\ enumerated by round $t$.
\alice\ wins the game if for some $t$, $1-||\Mb [t](\emptyset)||_1
\leq \frac{1}{a}(1-\m(A[t]))$.
\end{definition}
\begin{claim}\label{singclaim962}
Let $a,\varepsilon>0, \h n\in\omega$. 
Suppose
\alice\ has a winning strategy for the dynamic
$(a,\h n ,k )$-partial-\sided-game \ref{singdef963}
such that $\m(A)\leq 1-\varepsilon$.
Then for every $0<\delta\leq \varepsilon/2a$,
there is an $n$ such that
\alice\ has a winning strategy for
$(1-2\delta,n,k)$-partial-\sided-game \ref{singdef962} such that
$\m(A)\leq 1-a\delta$.
\end{claim}
\begin{proof}
Exactly  as Claim \ref{singclaim2}
since in Claim \ref{singclaim2} we merely use the
fact that $M_j[t]$ is supermartingale and nondecreasing.
\end{proof}

Thus
   Lemma \ref{singlem942} is reduced to
the following:
\begin{lemma}\label{singques960}
For every $a>0$, there is an $n\in\omega$ such that
\alice\ has a winning strategy for   dynamic $(a,n,k)$-partial-\sided-game \ref{singdef963}
such that $\m(A)<1$.
\end{lemma}

Again, we simplify  Lemma \ref{singques960} by imposing restrictions
on \baby's action.

\begin{definition}[Restricted dynamic $(a, \delta, n,k )$-partial-\sided-game]
 \label{singdef964}

At each round $t\in\omega$, \alice\ firstly enumerates a $\sigma\in 2^n$
(that has not been enumerated before);
then \baby\ present $p_j[t]$, $p_j[t]$-\sided\ supermartingale
$M_j[t]$ such that the following hold:
\begin{itemize}
\item $p_j[t]$ extends $p_j[t-1]$;
\item $||\Mb[t](\sigma)||_1\geq 1 $;
\item $||\Mb[t](\rho)||_1\leq 1+\delta$
for all $\rho\in 2^{\leq n}$;
\item $M_j[t]$ dominates $ M_j[t-1]$ on $2^{\leq n}$.
\end{itemize}
Let $A[t]$ denote the set of $\sigma$ \alice\ enumerated by round $t$.
\alice\ wins the game if for some $t$, $1-||\Mb[t](\emptyset)||_1
\leq \frac{1}{a}(1-\m(A[t]))$.
\end{definition}
By the same argument we used for Claim \ref{singclaim4} we have:

\begin{claim}\label{singclaim964}
Let $a,\delta>0, \t n\in \omega$.
Suppose \alice\ has a winning strategy for
the restricted dynamic $(a,\delta, \t n,k )$-partial-\sided-game \ref{singdef964}
such that $\m(A)< 1$.
Then  there is an $n\in\omega$ such that \alice\ has a winning
 strategy for
the  dynamic $(a/2,n,k )$-partial-\sided-game \ref{singdef963}
such that $\m(A)<1$.
\end{claim}

By Claim \ref{singclaim964},
to prove Lemma \ref{singques960},
it suffices to prove the following.
\begin{lemma}
\label{singlem943}

For every $a>0$,
there are  $n\in\omega, \delta>0$ such that
\alice\ has a winning strategy for the restricted
dynamic $(a,\delta,n, k )$-partial-\sided-game \ref{singdef964}
such that $\m(A)<1$.
\end{lemma}

In \S  \ref{singseccatching},
we reduce the restricted dynamic partial-\sided-game
to yet another game.

\subsection{Reducing to \variancee\ game}
\label{singseccatching}
As   commented below Lemma \ref{singlem941},
it suffices to show that \alice\ can force
\baby\ to allocate  large \variancee\ on
the set of strings she enumerates.
This gives rise to the following \variancee\ game.

\begin{definition}[Variance\ $(a,\Delta, m,k )$-partial-\sided-game]
\footnote{The variance game does not have restriction rule II for \baby.
Indeed, \alice\ can still win it without that rule (see \S  \ref{singsec100}).}
 \label{singdef965}

At each round $t\in\omega$, \alice\ firstly enumerates a $\sigma\in 2^m$
(that has not been enumerated before);
then \baby\ presents $p_j[t]$, $p_j[t]$-\sided\ supermartingale
$M_j[t]$ such that the following hold:
\begin{itemize}
\item $p_j[t]$ extends $p_j[t-1]$;
\item $||\Mb[t](\sigma)||_1\geq 1 $;
\item $M_j[t]$ dominates $ M_j[t-1]$ on $2^{\leq n}$.

\end{itemize}
Let $A[t]$ be the set of $\sigma$ \alice\ enumerated by round $t$.
\alice\ wins  if for some $t$,
\begin{itemize}
\item (type-(a)) either $1-||\Mb[t](\emptyset)||_1
\leq \frac{1}{a}(1-\m(A[t]))$;

\item (type-(b)) or $\variance( \Mb[t]|A[t])\geq \Delta$.

\end{itemize}
\end{definition}

Later, we will show that \alice\ does have a winning strategy
for the \variancee\   game \ref{singdef965}
such that $\m(A)<1$ (Lemma \ref{singlem966}).
Now we reduces Lemma \ref{singclaim964} to Lemma \ref{singlem966}.
Let $a,\Delta>0, m\in\omega$.
\begin{claim}\label{singclaim966}
Suppose \alice\ has a winning strategy for the
\variancee\ $(a,\Delta, m, k )$-partial-\sided-game \ref{singdef965}
such that $\m(A)<1$.
Then there are
$n\in\omega,\h\delta>0$ such that \alice\ has a winning strategy for
 restricted dynamic
$(a,\h\delta, n, k)$-partial-\sided-game \ref{singdef964} such that $\m(A)<1$.

\end{claim}
\begin{proof}

For each $l\in \omega$, we inductively define a
$l$-level-\variancee-strategy. For each $l$, it is
a strategy played on $2^{ml}$.
We show that for sufficiently large $l$ (depending on $a,\Delta,m,k$),
the strategy wins the restricted dynamic $(a,\h\delta, ml, k)$-partial-\sided\ game \ref{singdef964}
(where $\h\delta$ is sufficiently small depending on $m,\Delta$ and to be specified).

\begin{strategyn*}[$1$-level-\variancee-strategy]
This is simply the winning strategy for the
\variancee\ $(a,\Delta,m, k)$-partial-\sided-game \ref{singdef965}
except that when some $\rho$ receives \typeaattention\
(recall from (\ref{singeq11}) the definition of
\typeaattention),
\alice\ then enumerates $2^m\setminus [\rho]^\preceq$ and terminates.
Moreover, after reaching
the winning criterion of the  \variancee\ $(a,\Delta,m, k)$-partial-\sided-game
\ref{singdef965}\footnote{
Since they are actually playing the restricted dynamic game \ref{singdef964},
 now \baby\ even has to conform
to stronger restrictions
(i.e., restriction rule II). So the \variancee\ game winning
strategy still leads to a winning state (of the \variancee\ game).},
\alice\ keep enumerating whatever is left in $2^m$. (End of Strategy.)
\end{strategyn*}

The round $t$ \alice\ wins the
\variancee\ $(a,\Delta,m,k)$-partial-\sided-game \ref{singdef965}
is the round $\emptyset$ receives \emph{\potentialwin\ attention} (compare to (\ref{singeq14})).
 If no $\rho$ (in particular $\emptyset$ does not) receives \typeaattention,
it means \alice\ win the \variancee\ $(a,\Delta,m,k)$-partial-\sided-game
\ref{singdef965} in type-(b) way, namely
\begin{align}\label{singeq978}
\variance(\Mb[t]|A[t])\geq \Delta
\end{align}
(otherwise $\emptyset$ receives \typeaattention\ since by the hypothesis of Claim \ref{singclaim966}, $\m(A[t])<1$).
Since $\h\delta$ is sufficiently small (depending on $  m, \Delta $)
and $M_j[t]$ is nondecreasing,
we have
\begin{align}\label{singeq971}
&\variance(\Mb[\h t]|A[  t])\geq \Delta/2\text{ for all $\h t\geq t$;
 thus }\\  \nonumber
 &\variance(\Mb[\h t]|2^m)\geq 2^{-2m}\Delta\text{ for all $\h t\geq t$}.
\end{align}
We emphasize that these \variancee\ strategy are not
about winning the \variancee\ game \ref{singdef965},
but the restricted dynamic
$(a,\h \delta,m,k)$-partial-\sided-game \ref{singdef964}.
The \variancee\ strategy, if not interrupted by \typeaattention, will enumerate
all strings in the end.

\begin{strategyn*}[$l$-level-\variancee-strategy]
On $2^m$, \alice\ plays the $1$-level-\variancee-strategy.
But when the strategy tells her to enumerate $\rho$,
instead of enumerating $\rho$ directly,
\alice\ plays the $(l-1)$-level-\variancee-strategy
on $[\rho]^\preceq\cap 2^{ml}$.
Again, during this process (main game or sub-game),
if some $\rho\in 2^{\leq ml}$ receives \typeaattention,
\alice\ then enumerates $2^{ml}\setminus [\rho]^\preceq$
and terminates.
(End of strategy.)
\end{strategyn*}
We show that when $l$ is sufficiently large,
the $l$-level-\variancee-strategy
wins the restricted dynamic $(a,\h\delta,ml,k)$-partial-\sided-game \ref{singdef964}.
\vspace{0.3cm}

\noindent\textbf{Case (a):} Some $\rho$ receives \typeaattention.

This is the same as case (a) of the proof of Lemma \ref{singlem43} (see Fact \ref{singfactwinatt}).\vspace{0.3cm}

\noindent\textbf{Case (b):} Otherwise.

We derive a contradiction.
Note that if a sub-strategy over $\rho$ calls for \typeaattention,
then so does the main game
(since \typeaattention\ of $\rho$
 only depends on what happened in $[\rho]^\preceq$).
 Thus if some sub-strategy calls for \typeaattention,
the whole strategy terminates and \alice\ wins the restricted dynamic $(a,\h\delta,ml, k)$-partial-\sided-game \ref{singdef964}
(as we argued in Fact \ref{singfactwinatt}).
In other words, in this case, for every $\h l<l$, every $\rho\in 2^{m\h l}$,
\begin{align}\label{singeq973}
&
\text{the (sub) \variancee\ $(a,\Delta,m,k)$-partial-\sided-game \ref{singdef965}}\\ \nonumber
&\text{on $[\rho]^\preceq\cap 2^{|\rho|+m}$
is won in
a type-(b) way (i.e., (\ref{singeq978})).}
\end{align}

Also note that in a $l$-level-\variancee-strategy, when
\alice\ wins the main variance game (the one on $2^m$), it means
  $ [\t A[t]]^\preceq\cap 2^{ml}$ has been enumerated (where $\t A[t]\subseteq 2^m$ is the
  set of strings
   \alice\ consumed to win the \variancee\ game on $2^m$).
So
   (as (\ref{singeq971}) by restriction rule)
   \begin{align}\label{singeq972}
   &\text{once $\Mb[t]$ allocate a constant large variance on
   $\t A[t]$, }\\ \nonumber
   &\text{the variance won't   change too much ever since
   (i.e., (\ref{singeq971})).}
   \end{align}

 By the near end of the game
(when $\m(A[t])$ is close to $1$),
\begin{align}\label{singeq977}
&\text{ for sufficiently many $\rho\in \cup_{\h l<l}2^{m\h l} $
(say every $\rho\in \cup_{\h l <l/2}2^{m\h l}$),
}\\ \nonumber
 &\text{ \alice\ has won the \variancee\ game \ref{singdef965} on $[\rho]^\preceq\cap 2^{|\rho|+m}$.}
\end{align}

Combining (\ref{singeq971}), (\ref{singeq973}), (\ref{singeq972}), (\ref{singeq977})
(see case (b) of Lemma \ref{singlem43}) and
 letting $\t B_{\h l} = 2^{m\h l}$,
if for each $\rho\in \t B_{\h l-1}$ we let
$\t B_\rho = [\rho]^\preceq\cap \t B_{\h l}$,
there exists a round $t$ (near end of the game) such that:
\begin{align}\nonumber
\sum_{\h l<l}\int_{\rho\in \t B_{\h l}}\variance(\Mb[t]|\t B_\rho)
&\overset{}{\geq} 2^{-3m}\Delta\cdot l\text{ while }\\ \nonumber
\int_{\sigma\in 2^{ml}}||\Mb[t](\sigma)||_1&\geq 1-\varepsilon.
\end{align}
where  $\varepsilon$ is sufficiently small (say $\varepsilon=2^{-ml}$).
Since $\emptyset$ does not receive \typeaattention,
\begin{align}\nonumber
||\Mb[t](\emptyset)||_1\leq 1.
\end{align}
Therefore, by Claim \ref{singclaim5},
$$
\sum_{\h l<l}\int_{\rho\in \t B_{\h l}}\variance(\Mb[t]|\t B_\rho)
\leq C(k)(1+\varepsilon l)
$$
Since we can choose $\varepsilon= 2^{-ml} = o(2^{-3m}\Delta/C(k))$
and $l$ is sufficiently large, this is a
contradiction.
\end{proof}

\begin{remark}\label{singrem3}
Again, Claim \ref{singclaim966}
does not rely on \sided ness
but the supermartingale and nondecreasing of $M_j[t]$.
\end{remark}

By Claim \ref{singclaim966},
it remains to prove
\begin{lemma}\label{singlem966}
For every $a>0$,
there are $m\in\omega$, $\Delta>0$ such that
\alice\ has a winning strategy for
\variancee\ $(a,\Delta, m, k )$-partial-\sided-game \ref{singdef965} such that $\m(A)<1$.
\end{lemma}

\subsection{Winning the \variancee\ game}

\label{singsec100}
The difficulty to win the partial-\sided-game \ref{singdef962} is that
$M_j$ cooperate with each other. But in the \variancee\ game \ref{singdef965}
they
have to behave like  constant (on $A[t]$)
functions
 (otherwise \alice\ wins in the type-(b) way
 \footnote{We emphasize
that   $\Delta$ can be arbitrarily small
 depending on $a,m$.}),
 so there is no cooperation among
$M_j$. Therefore we could simply defeat them one by one.

We prove Lemma \ref{singlem966} by induction on $k$.
For $k=1$.
Let $a,\Delta>0$; let $m\in\omega$  be sufficiently large
(depending on $a $).

\begin{strategy}[$k=1$]\label{singstrategy0}
Enumerates each $\sigma\in ([1]^\preceq\cap 2^m)\setminus \{1\cdots 1\}$
 until
\begin{align}\label{singeq91}
&\text{$p_0[t](\emptyset)$ is defined; or }\\ \label{singeq90}
&\text{every $\sigma\in  ([1]^\preceq\cap 2^m)\setminus \{1\cdots 1\}$ is enumerated
(while $p_0[t]\uparrow$).}
\end{align}
If  (\ref{singeq91}) happens, \alice\ then enumerates $i$
where $p_0[t](\emptyset) =1-i$.

\noindent(End of Strategy.)
\end{strategy}
If (\ref{singeq91}) occurs,
 clearly \alice\ wins the
game since now $M_0[t](\emptyset)\geq 1$ while $\m(A[t])<1$
(since either $1\cdots 1\notin A[t]$ or $A[t]\cap [0]^\preceq=\emptyset$).
In case (\ref{singeq90}),  \alice\  wins
since
\begin{align}\nonumber
1-M_0[t](\emptyset)&\overset{p_0[t](\emptyset)\uparrow}{\leq}
1-M_0[t](1)\leq  2^{-(m-1)}\text{ while }\\ \nonumber
1-\m(A[t])&\geq 1/2.
\end{align}

Note that now we do not only have that Lemma \ref{singlem966} is true
 when $k=1$, but by what we have done in \S  \ref{singsecidr2} -
\ref{singseccatching}, we have
\begin{align}\nonumber
\bullet\ \ \text{Lemma \ref{singlem941} is true when $k=1$.}
\end{align}
In other words, by induction,
it suffices to prove that
\begin{align}\nonumber
\bullet\ \ \text{Lemma \ref{singlem941} (for $k=k-1)$
}\Rightarrow\text{ Lemma \ref{singlem966} (for $k=k$).}
\end{align}

For $k=k$; given $ a>0$.
Let $\varepsilon>0, c<1$ and $m\in\omega$
be such that $\varepsilon\approx 0,c\approx 1$ and
 \alice\ has a winning strategy
 for the   $( c, m/2,k-1)$-partial-\sided-game \ref{singdef962}
 such that $\m(A)\leq \varepsilon$. Let
  $ \Delta>0$ be sufficiently small (depending on $a,k$)
so that for some $\h \Delta>0 $:
 $2^{-m}(\h \Delta/2k)^2\geq \Delta$ and
 $1-c+2\h\Delta\leq (1-2\varepsilon)/a$.
The following strategy wins the \variancee\ $(a,\Delta,m,k)$-partial-\sided-game
\ref{singdef965} (for $k=k$).

\begin{strategyn*}[$k=k$]
Until reaching  type-(b) winning criterion, \alice\ do the following.
\begin{enumerate}[(1)]
\item In phase $1$: \alice\ plays the winning strategy
for the (scaled\footnote{\alice\ will play the winning strategy assuming type-(b) winning criterion
won't happen. If that is the case,
then let $c_0 = M_0[0](\sigma_0)$ where $\sigma_0\in 2^{m}$ is the first string
\alice\ enumerates, we have
$\sum_{0<j<k}M_j[t](\sigma)\geq 1-c_0-\h\Delta$
for each $\sigma\in 2^m$ \alice\ enumerates (since
$2^{-m}(\h \Delta/2k)^2\geq \Delta$). So she will play
a $(1-c_0-\h\Delta)$-scaled $( c,m/2,k-1)$-partial-\sided-game \ref{singdef962}.
It will be seen from the proof later that it doesn't matter
if $1-c_0-\h\Delta<0$.}) $( c,m/2,k-1)$-partial-\sided-game \ref{singdef962}
against $M_1,\cdots,M_{k-1}$ on $2^{m/2}$.
Let $\t A$ denote the set of strings in $2^{m/2}$ \alice\ has enumerated
when she wins (the $( c,m/2,k-1)$-partial-\sided-game \ref{singdef962}
against $M_1,\cdots,M_{k-1}$ on $2^{m/2}$).
\item In phase $2$, for each $\rho\in 2^{m/2}\setminus \t A$ (in whatever order),
play the winning strategy for the (scaled) $( c,m/2,k-1)$-partial-\sided-game \ref{singdef962}
against $M_0 $ on $[\rho]^\preceq\cap 2^m$.
\end{enumerate}
\noindent (End of Strategy.)
\end{strategyn*}

Assuming type-(b) winning criterion is not reached (otherwise we are done),
we prove that type-(a) winning criterion will be reached.
When \alice\ wins the $( c,m/2,k-1)$-partial-\sided-game \ref{singdef962}
against $M_1,\cdots,M_{k-1}$ on $2^{m/2}$,
by scaling,
$$\sum\nolimits_{0<j<k}M_j[t](\emptyset)\geq c(1-c_0-\h\Delta).$$
Here we are clearly using the crucial fact that $\sum_{0<j<k}M_j[t](\rho)\geq  1-c_0-\h \Delta$
for all $\rho\in 2^{m/2}$  enumerated by \alice\ in phase $1$ (since type-(b) winning criterion does not occur).
This is the advantage (for \alice) in the \variancee\ partial-\sided-game \ref{singdef965},
 which the original
partial-\sided-game \ref{singdef962} does not have (so this induction argument cannot prove Lemma \ref{singlem941}).
Let $t$ be the round everything is done.
Obviously,
for each $\rho\in \t A$, $M_0[t](\rho)\geq c_0-\h\Delta$ (since type-(b) winning criterion does not occur).
Meanwhile, for each $\rho\in 2^{m/2}\setminus \t A$, by scaling,
$M_0[t](\rho)\geq c(c_0-\h\Delta)$.
Thus $$M_0[t](\emptyset)\geq c(c_0-\h\Delta).$$
Therefore,
$\sum_j M_j[t](\emptyset)\geq c-2\h\Delta$
and the cost of \alice\ satisfies $\m(A)\leq 2\varepsilon$.
Thus we are done.

\begin{remark}\label{singrem-1}
The induction step, does not rely on partial-\sided ness
but only the nondecreasing and supermartingale properties of $M_j[t]$.
In combination with remarks \ref{singrem1}, \ref{singrem2}, \ref{singrem3},
strategy \ref{singstrategy0} is the only part of
the proof of Lemma \ref{singlem941} concerning partial-\sided ness.
\end{remark}

%% file: generalizebetting3combine.tex
\section{Subclasses of left-c.e. supermartingale}
\label{singsecgen}

We address the question: is there a proper subclass
of left-c.e. supermartingale defining 1-randomness.
First, we need to formalize the question.
A   \emph{ \supermartingaleapp }
 is a sequence $M[\leqslant t] = (M[0],\cdots,M[t])$
 of supermartingales.
We use $\mcal{M}$ to denote a set of \supermartingaleapp s.
An \emph{$\mcal{M}$-gale} is an infinite sequence $(M[t]:t\in\omega)$
such that $M[\leqslant t]\in \mcal{M}$ for each $t\in\omega$
and   $\lim_{t\rightarrow\infty} M[t](\sigma)$ exist for all $\sigma\in 2^{<\omega}$;
it is \emph{computable} iff $(M[t]:t\in\omega)$ is computable.

 \begin{example}
 For muchgale, let $\mcal{M}_l$ be the class of nondecreasing
 \supermartingaleapp s\
 $M[\leqslant t]$ such that
 for some $i<l$, $M[\h t]$ is $(l,i)$-\betting\ for all $\h t\leq t$.
 Let $\mcal{M} = \cup_l \mcal{M}_l$, then muchgale is computable $\mcal{M}$-gale.
 For kastergale,
 let $\mcal{M}$ be the class of nondecreasing supermartingale approximations
 $M[\leqslant t]$ such that
 for every $\sigma\in 2^{<\omega}$,
 if there is a $\h t\leq t$ such that $M[\h t](\sigma0)>M[\h t](\sigma 1)$,
 then for every $\h t\leq t$, $M[\h t](\sigma0)\geq M[\h t](\sigma 1)$.
 Then kastergale is computable $\mcal{M}$-gale.
\end{example}

Game \ref{singdef2} can be defined with respect to
a class $\mcal{M}$ of \supermartingaleapp s.
\begin{definition}[$(c,n,  k)$-$\mcal{M}$-game]
\label{singdefgen0}
At each round $t\in\omega$, \alice\ firstly enumerates a $\sigma\in   2^n$;
then \baby\ presents $M_j[t] $ with  $ M_j[\leqslant t] \in\mcal{M}$ for each $j<k$ such that:
\begin{itemize}
\item for every $\sigma\in A[t]$,
 $|| \Mb[t](\h\sigma)||_1\geq 1 $ for some $\h\sigma\preceq\sigma$;
\end{itemize}
\alice\ wins the game if for some $t$, $||\Mb[t](\emptyset)||_1\geq  c$.
\end{definition}

\def\shrinkinginvariant{subsequence-closed}
 We say $\mcal{M}$ is \emph{nondecreasing} if for each $ M[\leqslant t] \in\mcal{M}$,
$M[\h t]$ dominates $M[\h t-1]$ for all $1\leq \h t\leq t$.
We say $\mcal{M}$ is \emph{\homogeneous} iff:
for every $\rho\in 2^{<\omega}$, let $h_\rho$ denote the isomorphism
$\sigma\mapsto \rho\sigma$, we have
$ M[\leqslant t]\in\mcal{M}$ implies $ M[\leqslant t]\circ h_\rho\in\mcal{M}$.
We say $\mcal{M}$ is \emph{\shrinkinginvariant} iff
for every $M[\leqslant   t]\in\mcal{M}$,
every $ t_0<\cdots<t_{s-1}\leq t$,
$(M[t_0],\cdots,M[t_{s-1}])\in\mcal{M}$
\footnote{\label{singfootnote0} The \shrinkinginvariant\ together with
$\Pi_1^0$-class of $\mcal{M}$ ensures that the winning strategy of \alice\ is computable,
which is needed when applying the winning strategy in,
say, the proof of Theorem \ref{singth3}.
To this end,
it suffices that \alice\  computably knows all the possible action of \baby\ at round $t$
given the game history.
Note that $\Pi_1^0$-class of $\mcal{M}$
is not enough.
Because when \alice\ plays against computable $\mcal{M}$-gales
and  enumerates a string  (at time $t$),
the valid action of \baby\ may come very late, say at time $\h t$
where $\h t$ is much larger than $t$.
For example, consider the class $\{M^*\}$ where $M^*$ is the universal
left-c.e. supermartingale. For this class,
\alice\ has a winning strategy for the $(c,n,1)$-$\mcal{M}$-game
played at $[\rho]^\preceq\cap 2^{\leq n+|\rho|}$
by enumerating one string in $[\rho]^\preceq\cap 2^{n+|\rho|}$, provided $M^*(\rho)<c$. But the winning  strategy is not computable
(uniformly in $\rho$) as
the class is not \shrinkinginvariant.
}.
We say $\mcal{M}$ is \emph{scale-closed} iff for every $M[\leqslant t]\in\mcal{M}$,
every $c>0$, $cM[\leqslant t]\in\mcal{M}$.
Abusing terminology, we say $\mcal{M}$ is \emph{$\Pi_1^0$-class} iff for every $t\in\omega$,
the set of $M[\leqslant t]\in\mcal{M}$ is a $\Pi_1^0$-class uniformly in $t$.

\begin{remark}\label{singremhomo}
The \stronghomogeneity\ property is implicitly used in all  ingredients
(as all those finite games are defined on $[\emptyset]^\preceq\cap 2^{\leq n}$ instead of
$[\rho]^\preceq\cap 2^{\leq |\rho|+n}$).
However, this property is not essential, but for notation convenience
since those finite games can be defined on $[\rho]^\preceq\cap 2^{\leq |\rho|+n}$
and all ingredients (nesting, dynamic goal, restriction and reduction to variance game)
can be generalized accordingly.
\end{remark}

Let classes of \supermartingaleapp\ $(\mcal{M}_l:l\in\omega)$ be $\Pi_1^0$-class (uniformly in $l$),
\shrinkinginvariant,
  scale-closed, nondecreasing and homogeneous with $\mcal{M}_l\subseteq \mcal{M}_{l+1}$ for all $l\in\omega$;
and $\mcal{M} = \cup_l \mcal{M}_l$.
As  the proof of Theorem \ref{singth3} using Lemma \ref{singlem43},
\begin{claim}\label{singeqgen1}
Suppose for every  $l\in\omega$,
every $ c<1, \varepsilon>0$ and $k\in\omega$,
 there is an $n\in\omega$, such that
 \alice\ has a winning strategy
for   $(c,n,  k)$-$\mcal{M}_l$-game \ref{singdefgen0} such that
$\m(A)\leq \varepsilon$.
Then there is a non-1-random real $x\in 2^\omega$
on which no computable $\mcal{M}$-gale succeeds.
\end{claim}
\begin{remark}
We note that \shrinkinginvariant\ and $\Pi_1^0$-class of $\mcal{M}_l$
are essential for Claim \ref{singeqgen1} (see footnote \ref{singfootnote0}).
On the other hand, homogeneity is only for convenience (see remark \ref{singremhomo}).
Actually, nondecreasing may not be essential
for Claim \ref{singeqgen1}. But
discussing this issue does seem worth wile.
\end{remark}
As commented in remarks \ref{singrem1}, \ref{singrem2}, \ref{singrem3} and \ref{singremhomo}
(that nesting, dynamic goal, restriction and reducing to \variancee\ game only concern
nondecreasing, supermartingale  and \homogeneity\ of $\mcal{M}$),
the winning strategy of $\mcal{M}$-game \ref{singdefgen0} (with arbitrarily small cost)
is reduced to the winning strategy of the corresponding \variancee\
game (with a cost smaller than $1$).
\begin{definition}[\variancee\ $(a,\Delta ,m,k)$-$\mcal{M}$-game]
\label{singdefgen2}
At each round $t\in\omega$, \alice\ firstly enumerates a $\sigma\in  2^m$;
then \baby\ presents $M_j[t]$ with $M_j[\leqslant t] \in\mcal{M}$ for each $j<k$ such that
\begin{itemize}
\item for every $\sigma\in A[t]$,
 $|| \Mb[t]( \sigma)||_1\geq 1 $.
\end{itemize}
\alice\ wins the game if for some $t$,
\begin{itemize}
\item (type-(a)) either $1-|| \Mb[t](\emptyset)||_1\leq  \frac{1}{a} (1-\m(A[t]))$;
\item (type-(b))  or $\variance(\Mb[t]|A[t])\geq \Delta$
\end{itemize}
\end{definition}

By remark \ref{singrem-1},
a winning strategy for \variancee\ game \ref{singdefgen2} is
reduced to the special case $k=1$ and
that reduction only depends on nondecreasing, supermartingale
and \homogeneity\ so:
\begin{claim}\label{singgenclaim0}
Let $\mcal{M}$ be nondecreasing and \homogeneous and
suppose for every  $a>0$,
there are $m\in\omega$, $\Delta>0$
such that
\alice\ has a winning strategy
for
  \variancee\ $(a,\Delta ,m,1)$-$\mcal{M}$-game \ref{singdefgen2}
such that $\m(A)<1$.
Then for
every $ c<1, \varepsilon>0$ and $k\in\omega$,
there is an $n\in\omega$ such that
 \alice\ has a winning strategy
for   $(c,n,  k)$-$\mcal{M}$-game \ref{singdefgen0} such that
$\m(A)\leq \varepsilon$.
\end{claim}

We already demonstrated how to defeat a single
kastergale (in the \variancee\ game) in \S  \ref{singsec100}.
Thus, to prove Theorem \ref{singth5},
it suffices to show how to defeat a single muchgale.
However, this is obvious: for the class of $(l,i)$-betting
supermartingales, \alice\ could simply enumerate the set $\{\sigma\in 2^n:\sigma(i) = 0\}$,
forcing $M[t](\emptyset)\geq 1$.
By Claims \ref{singeqgen1}, \ref{singgenclaim0}, we get: \singthunion*

\begin{remark}
On the other hand, if \alice\ could not even
win \baby\ (in the \variancee\ game \ref{singdefgen2}) when $k=1$,
by the definition of \variancee\ game \ref{singdefgen2},
it is readily seen that some member
(the one \baby\ used against \alice) of $\mcal{M}$
is almost as flexible as a general left-c.e. supermartingale.
This strongly supports conjecture \eqref{singconj0}:
if a natural subclass of left-c.e. supermartingale  defines 1-randomness, then a single member
of that class can do so.
\end{remark}

\section{Conclusion}\label{FelK1JLCyK}
We have shown that a wide class of restricted enumerable strategies including kastergales and muchgales
cannot define 1-randomness, answering questions by Kastermans and Hitchcock as well as generalizing Muchnik's paradox.
In this process we developed a general methodology, which
strongly supports the thesis that the universal \lce supermartingales
is {\em irreducible} in the class of \lce supermartingales, in the sense that no restricted class of \lce supermartingales succeed
on all reals on which the universal \lce supermartingales succeeds.
We have noted that simple special cases of this result, such has the case of single-sided betting strategies, were immune to existing methods which were largely recursion-theoretic.

One drawback of our methodology is that it does not give any information about the effective Hausdorff dimension of the constructed non-1-random reals.
In the same way that non-1-randomness is related to the success of \lce supermartingales,
effective Hausdorff dimension relates to the growth-rate of the capital in betting strategies:
\citet{Lutz:00,Lutz:03} showed that the reals $x$ with effective Hausdorff dimension $<1$ are exactly the reals on which a
\lce supermartingale succeeds with exponential growth of the capital.
In this sense, effective Hausdorff dimension can gauge the power of a subclass of \lce betting strategies in a more refined way. by
examining the complexity of the reals $x$ where they fail.
For example, in the case of the class $\MM$ of
effective mixtures of single-sided martingales,
\citet{barmpalias2020monotonous}  showed that there exists a real $x$ of effective Hausdorff dimension
$\dim_H (x)=1/2$
such that no betting strategy in $\MM$ succeeds on $x$; moreover, this is optimal in the sense that some betting strategy in $\MM$ succeeds on
all reals $y$ with $\dim_H (y)<1/2$. For the case of muchgales, a similar result was obtained by \citet{muchpar22}.
However, the question for kastergales or even single-sided \lce supermartingales is wide-open:
\begin{equation*}
\parbox{11cm}{Is there a real with effective Hausdorff dimension $<1$ such that no single-sided \lce supermartingale (alt.\ no kastergale)
succeeds on it?}
\end{equation*}
The fact that this question remains open indicates the power of single-sided enumerable betting strategies is over other forms of restricted betting strategies such as muchgales.